\documentclass[11pt]{amsart}
\allowdisplaybreaks

\usepackage{geometry} 
\usepackage{framed} 

\usepackage{tikz}
\tikzset{every picture/.style={line width=1pt}} 

\usepackage{setspace}
\setdisplayskipstretch{2.5}
\usepackage{amsthm} 
\usepackage{amsmath}
\usepackage{amssymb}
\usepackage{mathrsfs}
\usepackage{bm}
\usepackage{subfigure}
\usepackage[hidelinks]{hyperref}
\usepackage[noabbrev]{cleveref}

\usepackage{amsfonts} 

\newcommand\C{\mathbb{C}}

\renewcommand\H{\mathcal{H}}

\newcommand\R{\mathbb{R}}

\newcommand\Z{\mathbb{Z}}

\newcommand{\ie}{\mathrm{ie}}

\newcommand\vphi{\varphi}
\renewcommand\phi{\vphi}
\newcommand\eps{\varepsilon}

\usepackage{stmaryrd} 

\newcommand\dist{\textnormal{dist}}

\newcommand\del{\partial}

\renewcommand{\div}{\textnormal{div}}

\newcommand{\CC}{\mathbb C}

\newcommand{\RR}{\mathbb R}
\newcommand{\ZZ}{\mathbb Z}
\newcommand{\calA}{{\mathcal A}}

\newcommand{\calC}{{\mathcal C}}

\newcommand{\calH}{{\mathcal H}}
\newcommand{\calI}{\mathcal I}

\newcommand{\calP}{{\mathcal P}}

\newcommand{\calT}{{\mathcal T}}
\newcommand{\bbT}{\mathbb T}

\newcommand{\calU}{{\mathcal U}}

\newcommand{\Si}{\Sigma}

\makeatletter
\def\@tocline#1#2#3#4#5#6#7{\relax
  \ifnum #1>\c@tocdepth 
  \else
    \par \addpenalty\@secpenalty\addvspace{#2}%
    \begingroup \hyphenpenalty\@M
    \@ifempty{#4}{%
      \@tempdima\csname r@tocindent\number#1\endcsname\relax
    }{%
      \@tempdima#4\relax
    }%
    \parindent\z@ \leftskip#3\relax \advance\leftskip\@tempdima\relax
    \rightskip\@pnumwidth plus4em \parfillskip-\@pnumwidth
    #5\leavevmode\hskip-\@tempdima
      \ifcase #1
       \or\or \hskip 1em \or \hskip 2em \else \hskip 3em \fi%
      #6\nobreak\relax
    \dotfill\hbox to\@pnumwidth{\@tocpagenum{#7}}\par
    \nobreak
    \endgroup
  \fi}
\makeatother

\newtheoremstyle{newtheoremstyle}
{3pt}
{3pt}
{\itshape}
{\parindent}
{\bfseries}
{.}
{0.5em}
{} 

\newtheoremstyle{newtheoremstyledefn}
{3pt}
{3pt}
{}
{\parindent}
{\bfseries}
{.}
{0.5em}
{} 

\theoremstyle{newtheoremstyle}
\newtheorem{theorem}{Theorem}
\newtheorem*{theorem*}{Theorem}
\newtheorem{lemma}[theorem]{Lemma}
\newtheorem{prop}[theorem]{Proposition}
\newtheorem{corollary}[theorem]{Corollary}

\theoremstyle{newtheoremstyledefn}
\newtheorem{defn}[theorem]{Definition}

\newtheorem{remark}{Remark}

\numberwithin{equation}{section} 
\numberwithin{theorem}{section}

\usepackage{fancyhdr}
\begin{document}

\title{Minimal Surfaces with Stratified Branching Sets}

\author{
	Federico Franceschini
	\and
	Rafe Mazzeo 
        \and
        Paul Minter
}

\address{\textnormal{Department of Mathematics, Stanford University, Building 380, Stanford, CA 94305, USA}}
\email{ffede@stanford.edu, rmazzeo@stanford.edu}
\address{\textnormal{Department of Pure Mathematics and Mathematical Statistics, University of Cambridge, Wilberforce Road, Cambridge, UK, CB3 0WB}}
\email{pdtwm2@cam.ac.uk}
\address{\textnormal{}}
\email{}

\begin{abstract}
Inspired by the Taubes--Wu construction of $\calC^{1,\alpha}$ two-valued harmonic functions by the use of symmetry, we construct minimal surfaces 
with stratified branching sets as graphs of $\calC^{1,\alpha}$ two-valued functions. 

We give three constructions. The first is perturbative and produces branched minimal submanifolds in arbitrary codimension as two-valued graphs
over the $n$-ball or, slightly more generally, over the product of $B^n$ with a torus $\bbT^N$, parametrized by boundary data which is required to be small
in a suitable norm. The second uses barrier methods together with a reflection argument to produce branched stable minimal hypersurfaces,
again as two-valued graphs over the unit $n$-ball or $B^n \times \bbT^N$, parametrized by boundary data which now can be large. Finally, using 
bifurcation theory, we produce compact minimal submanifolds with similarly stratified branching sets in an ambient space $S^n \times \RR$ with 
a suitable (analytic) warped product metric.  

These examples give minimal submanifolds with novel frequency values and whose branching sets have non-trivial deeper strata.  While the main constructions are fairly elementary, they rely
on the use of precisely tailored (and somewhat non-standard) function spaces, combined with a regularity theory which provides full
asymptotic expansions around the branching sets. 
\end{abstract} 

\maketitle

\section{Introduction} 

\subsection{Overview} 
It is well-known that a minimal surface $M$ can exhibit branch point singularities, i.e., singular points where $M$ has a (unique) tangent plane 
of multiplicity $Q>1$, yet $M$ is not locally the union of $Q$ distinct minimal graphs over its tangent plane. The set of all such points 
where this occurs is called the \emph{branching set}.   In this paper we are concerned entirely with branch points of 
multiplicity two, i.e., where $Q=2$.

Branched minimal surfaces are known to exist; a standard example is the minimal surface $\{(z,w)\in \C^2:z^2=w^3\}$, which is represented 
by the graph of the (average-free) two-valued $\calC^{1,1/2}$ function $u(w):=\pm w^{3/2}$, and has a branch point at $0$. Similar codimension 
$2$ examples can be constructed as algebraic subvarieties in $\C^n$.   

\begin{figure}[b]
\centering
\includegraphics[width=0.35\linewidth]{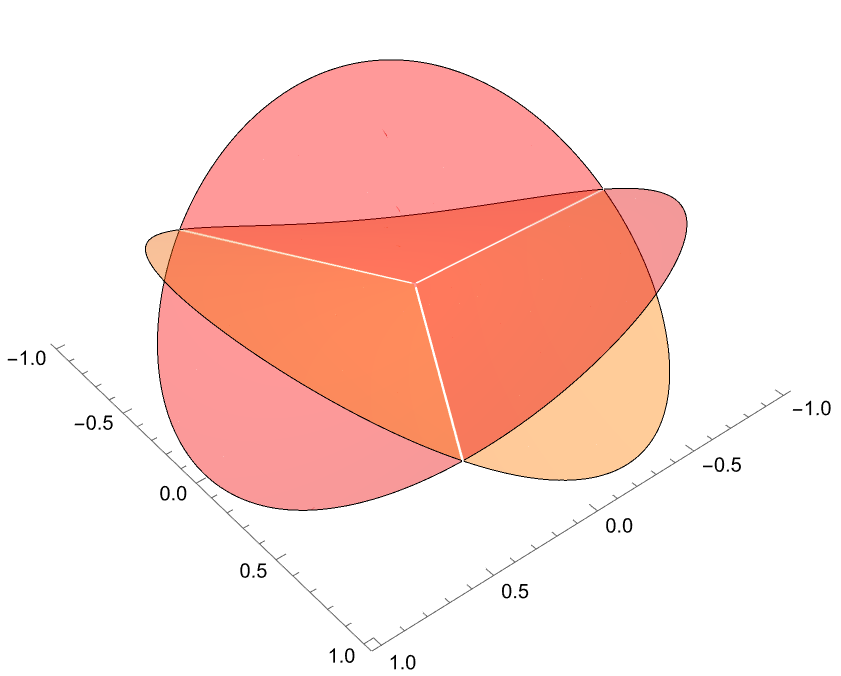}
\caption{\footnotesize Plot of $\text{Re}(x+iy)^{3/2}.$ The graph cannot be disentangled into two distinct smooth graphs around the origin: it is a branch point.}
\label{fig:3/2}
\end{figure}
 
These examples already indicate what turns out to be the typical
behavior of branch points with multiplicity $2$ (cf.~\cite{KW21}). Indeed, around any branch point $p\in M$, we define the \emph{frequency value} $\nu \geq 1$ by 
\[
\dist (M\cap B_r(p), T_pM\cap B_r(p)) \sim r^\nu\qquad \text{ as } r\to 0,
\]
if it exists. Here $B_r(p)$ is a ball in the ambient space and $\dist$ is, say, the Hausdorff distance.  As motivated by the complex
examples above, and justified more fully later in this paper, we call $\{\tfrac{3}{2}, \tfrac52, \tfrac72, \tfrac{11}{2},\ldots\}$ the 
set of {\it $2$-dimensional frequency values}. 

There are many codimension $1$ examples, constructed as minimal immersions in $\R^3$ using appropriate Weierstrass data. For these, the 
branching sets consist of a discrete set of points and the frequencies at these points necessarily lie in the set of $2$-dimensional frequency values above.
More generally, Simon and Wickramasekera \cite{SW07} constructed a large family of branched stable minimal hypersurfaces in $\RR^{n+1}$ 
as graphs of two-valued $\calC^{1,\alpha}$ functions (see also \cite{Kru19, Ros10}). In these cases, the branch set is $\{(0,0)\}\times \R^{n-2}$ 
and the structure of the surface around all the branch points is locally modelled on the graphs of $\pm\text{Re}((x_1+ix_2)^{k/2})$ for $k\geq 3$ odd.
In particular the set of frequencies is the 2-dimensional one.

To the authors' knowledge, all previously known codimension one branched minimal submanifolds have ``simple'' branch sets which consist 
of a smooth submanifold (of codimension $2$ inside $M$) and branch points having a 2-dimensional frequency value. The present paper
constructs three classes of examples of minimal surfaces of any dimension $n\ge 2$ and any codimension $\ge1$, that:
\begin{itemize}
\item[(a)] Possess a stratified branching set of dimension $n-2$ (in particular, the branching set is not a smooth submanifold),
\item[(b)] Have at least some branch points with (explicit) non-$2$-dimensional frequency values. This is especially interesting 
in the codimension one case. 
\end{itemize}

\subsection{Main results}
All our examples arise as the graphs of (average-free) two-valued functions of class $\calC^{1,\alpha}$ over $B^n$, $B^n \times \bbT^N$, or 
$S^n \times \bbT^N$, where $n\geq 3$, $N\geq 0$, and $\bbT^N$ is the flat $N$-dimensional torus. 
Let us describe the three classes of examples constructed here.

The first concerns perturbations of the disk $B^3\times\{0\}\subset\R^4$ with multiplicity $2$. The key to this construction is that the deformations 
preserving minimality are modelled on two-valued \textit{harmonic} functions on $B^3$ of class $\calC^{1,\alpha}$ which themselves have stratified 
branching sets consisting of a finite union of rays emanating from the origin. The frequency at any point along these rays is one of the $2$-dimensional 
values, but significantly, it jumps at the origin to a value which is strictly greater than $3/2$ and is not a half integer. 
(In fact, this frequency value is the positive solution to the quadratic equation $\nu(\nu+1)=\lambda_1(T),$ where $\lambda_1(T)$ is the 
first Dirichlet eigenvalue of a face of a regular spherical tetrahedron.)

The existence of such two-valued harmonic functions is a recent result by 
Taubes and Wu \cite{TW20}. Their examples are constructed by imposing a certain discrete symmetry; we refer to \cite{He,FS25} for 
different examples and the forthcoming \cite{TW26} for certain higher dimensional cases. We expand on this, first generalizing 
the idea to $B^n$ for $n \geq 3$, and then showing that the same construction works for minimal perturbations of $B^n \times \bbT^N$, 
where the discrete group acts only on the first factor. In the case $n=2$ we recover the average-free examples of  \cite{SW07}.
All of this is the content Section \ref{sec:perturbative} -- see Theorem~\ref{thm:main-perturbative}.  
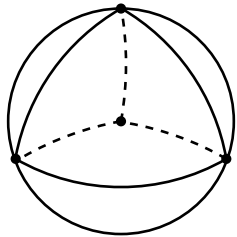
\begin{figure}[h]
    \centering
    \begin{tikzpicture}
        \draw (0,0) circle (1.5);
        \filldraw (0,1.5) circle (0.05);
        \filldraw (-1.4,-0.5) circle (0.05);
        \filldraw (1.4,-0.5) circle (0.05);
        \filldraw (0,0) circle (0.05);
        \draw [dashed] (0,0) arc (-10:10:4.4);
        \draw (-1.4,-0.5) arc (-120:-60:2.8);
        \draw [dashed] (1.4,-0.5) arc (60:80:4);
        \draw [dashed] (0,0) arc (100:120:4);
        \draw (0,1.5) arc (120:169:2.9);
        \draw (1.4,-0.5) arc (10:60:2.9);
    \end{tikzpicture}
    \caption{\footnotesize The spherical tetrahedron, which is determined by the spherical projection of the standard 3-simplex in $\R^3$.}
    \label{fig:tetrahedron}
\end{figure}

Section \ref{sec:perturbative} also contains a review and a streamlined proof of the perturbation theorem of Caffarelli--Hardt--Simon \cite{CHS}, 
as well as its generalization by Smale \cite{Smale93}. In this section we also provide background on the emerging theory of $\ZZ_2$-harmonic 
functions, i.e., two-valued harmonic functions of class $\calC^1$. 

The second construction, which holds only in codimension $1$, involves a more traditional approach via barriers 
in the spirit of Jenkins and Serrin \cite{JS}. This produces two-valued functions over $B^n$ (or $B^n \times \bbT^N$) satisfying the minimal 
surface equation and which branch over the same union of rays in the ball if $n=3$ (and appropriate linear sectors if $n > 3$), but which
are no longer perturbative so the graph functions $u$ may have large norm. This is the content of Theorem~\ref{thm:main}. We remark that the proof of Theorem~\ref{thm:main} is 
essentially self-contained.

Both of the above sets of examples are obtained by solving boundary value problems at the outer boundary $\del B^n \times \bbT^N$.
It is natural to ask for compact examples as well, and we show that these too can be constructed, but now using bifurcation theory.
These arise as submanifolds of a product $S^n \times (-1,1)$ with a warped product metric $g = dt^2 + f(t)^2\, g_{S^{n}}$
for suitable warping functions $f(t)$. These types of manifolds can be `planted' inside a larger compact manifold.  These examples are found as
the result of applying some standard bifurcation theorems to the family of minimal surface operators associated to
the family of metrics $dt^2 + f(\lambda t)^2 g_{S^n}$ as the parameter $\lambda$ increases. This is Theorem~\ref{thm:bifurcation2}.
 
\subsection{Comments} There are several morals to be drawn from all of this.  The first is that, given some necessary analytic background, the construction of perturbative branched minimal submanifolds
is remarkably simple. In particular, this argument turns out to be simpler than the
classical result of Caffarelli--Hardt--Simon \cite{CHS} regarding minimal perturbations of minimal cones, in the sense that it is unobstructed.  This suggests that,
somewhat unexpectedly, the local theory of branched minimal submanifolds might be somewhat simpler than that for isolated
singularities.  (On the other hand, the extension of \cite{CHS} to construct minimal perturbations of a product $C(Z) \times S^1$
of an embedded cone with a single simple circle turns out to require enormously more subtle analysis, see \cite{MP}.) 

However, all three of our results require the use of rather carefully crafted function spaces which reflect the
structure of the branching sets, and then knowledge of the precise mapping properties of the Laplacian on these function spaces.
Fortunately, this can be drawn from known literature in geometric microlocal analysis, and we have made efforts to explain some 
relevant parts of this analytic background which may be less familiar to some readers.   The second result (Theorem~\ref{thm:main}) 
is in this analytic-sense the most classical, i.e., the primary existence theorem can be carried out in standard spaces and using barrier arguments.
However, for all three main theorems, the methods we employ can be used to prove very sharp regularity statements for
the minimal submanifolds. These statements assert that the two-valued graph function $u$ admits a 
complete {\it polyhomogeneous} asymptotic expansion at the branching sets (see Definition ~\ref{def:polyhomog}).  This is a considerable refinement over
the more elementary gradient bounds, for example, and provides the sharp value of the frequency at the branch points with non-2-dimensional frequency. 

We conclude this introduction with a general remark that the existence of two-valued harmonic functions which decay to order
$3/2$ at their branching sets, or similarly the existence of two-valued functions satisfying the minimal surface equation
with this extra rate of decay, is quite delicate analytically.  It is akin to solving a free boundary problem where the branch points 
behave as a codimension two free boundary, where one imposes both Dirichlet and Neumann conditions.  The overdetermined
nature of this problem indicates that the location of the branching set is very constrained. By contrast, requiring only vanishing 
of order $1/2$, the location of this set would be essentially unconstrained.  The Taubes--Wu argument using
the group symmetry provides a vast simplification and short-cut to existence results.  It does not allow, however, for the construction of
branching sets which are more `non-linear'. We leave this as a challenge for the future.
 
\subsection{Acknowledgements} This research was conducted during the period PM served as a Clay Research Fellow at Stanford University.

\section{Perturbative Examples}\label{sec:perturbative}
In this section, we construct minimal submanifolds of any dimension $n\geq 2$ and codimension $k\geq 1$ which have a stratified branching set, i.e., the branch set is the union of a finite number of smooth submanifolds each having dimension $n-2$ or less, but is not itself a smooth submanifold. 
Our construction combines a symmetry argument due to Taubes--Wu \cite{TW20} with an updated version of the classical perturbative result of 
Caffarelli--Hardt--Simon \cite{CHS}.  

To illustrate and motivate our methods, we first review the original result from \cite{CHS}, or rather an extension of this result proved in \cite{Smale93}.
We take the opportunity to recast the proof slightly to identify clearly the necessary analytic tools. Later results in the paper rely on generalizations of these tools and will be described in due course.

In the next subsection, we briefly review the setting of conic spaces and conic elliptic operators and describe the salient facts. 
This then leads to a streamlined proof of the perturbative existence results from \cite{CHS, Smale93}. In the following subsection we review the theory of $\ZZ_2$-valued harmonic functions and the symmetry method of Taubes--Wu. The final subsection puts all of this together and we construct 
the perturbative families of minimal submanifolds with stratified branching sets.

\subsection{Elliptic conic theory and the Caffarelli--Hardt--Simon theorem}  
The setting of manifolds with isolated conic singularities is well-known. For simplicity, we describe only the class of conic \emph{embedded
sub}manifolds (rather than abstract Riemannian manifolds with conic singularities). 
Briefly, a \emph{conic submanifold} $X$ in an ambient Riemannian manifold $Z$ consists of a smoothly embedded submanifold
together with a finite collection of \emph{conic points} $p_1, \ldots, p_m$, such that in a neighborhood $\calU_\ell$ of each $p_\ell$
there exists a coordinate system $x$ centered at $p_\ell$, a standard warped product cone $C_\ell$ in these coordinates,
and a section $u_\ell$ of the normal bundle of $C_\ell$ which decays faster than linearly as $r = |x| \to 0$, such that $X \cap \calU_\ell$
is the graph of this normal section over $C_\ell$.   If $\dim Z = N$, we introduce polar (spherical) coordinates $x = r y$, $y \in S^{N-1}$,
and require that a warped product cone is a subset of the form $\{(r, y): 0 \leq r \leq 1,\ y \in \Sigma_\ell \subset S^{N-1}\}$. 
We require the graph function $u_\ell$ to decay like $r^\mu$ for some $\mu > 1$ and also impose certain estimates on at least a finite
number of its derivatives.  By virtue of this decay rate, the tangent cone to $X$ at $p_\ell$ is precisely the cone $C_\ell \subset 
\RR^N = T_{p_\ell} Z$ (extended to all $r \geq 0$). 

Now fix a compact $(k+1)$-dimensional conic minimal submanifold-with-boundary $X$ (where all the conic points lie in the interior). Minimality is with respect to
the given metric $g$ on the ambient space $Z$.
There are many examples, including the ones obtained by perturbation methods in \cite{CHS} and more obtained by bridging 
these together as in \cite{Smale, Smale93, Dimler}. More generally, this class includes all minimal submanifolds satisfying the hypotheses of the 
famous result of Simon \cite{Simon} concerning uniqueness of (multiplicity one) tangent cones with isolated singular points. 

Let us now examine the Jacobi operator of this minimal submanifold $X$ in any one of the neighborhoods $\calU_\ell$;
for simplicity we omit the index $\ell$. Suppose first that $C \subset \RR^N$ is an exact warped product minimal cone with vertex at $0$,
and write $\Sigma = C \cap S^{N-1}$.  This link $\Sigma$ is necessarily a (smooth) minimal submanifold of the sphere and, if $L_\Sigma$ denotes 
its corresponding Jacobi operator, then the Jacobi operator of $C$ in $\RR^N$ takes the form
\begin{equation}
L = \del_r^2 + \frac{k}{r} \del_r + \frac{1}{r^2}L_\Sigma
\label{JacobiX}
\end{equation}
This is an example of a \emph{conic elliptic operator}. The analytic properties of such operators are well understood.  We discuss first the natural function
spaces on which they act and then the salient Fredholm and regularity properties, specialized to this particular operator $L$. 

Observe that $r^2 L$ is an `elliptic combination' of the basic vector fields $r\del_r$ and $\del_{y_j}$ with smooth coefficients
where, as before, the $y_j$ are coordinates on $\Sigma$. These are called the \emph{$b$-vector fields}.  In the more general setting when $X$ 
is only asymptotically conic, the Jacobi operator takes this form to leading order near each of the conic points. In other words, 
$r^2 L$ can then be expressed as the sum of this leading order model part plus a (possibly second order) remainder term of the form
$ r \sum_{j,\beta}  a_{j \beta} (r\del_r)^j \del_y^\beta$ which vanishes to one order higher as $r\to 0$.

The key idea is that we consider function spaces and related constructions adapted to the particular degeneracies exhibited by these
$b$-vector fields. Thus, we first define the $b$\emph{-Hölder spaces}.  Noting that $r\del_r$ and $\del_{y_j}$ are invariant under radial 
dilation $(r,y)\mapsto (\lambda r,y)$, it is natural to use function spaces which are similarly dilation-invariant. To this end, decompose $C$ 
(and, more generally, $X \cap \calU_\ell$) into the union of (overlapping) dyadic annuli 
\[
X \setminus \{0\} = \bigcup^\infty_{j=1}A_j \qquad \text{where }A_j := \{(r,y):2^{-j-1}\leq r\leq 2^{-j+1}\}.
\]
Given any function $u$ on $X$, define the sequence of functions $u_j$ on $A_1$ by $u_j = D_{2^j} (u|_{A_j})$, where $D_\lambda$ is the dilation $D_\lambda u(r,y) :=
u(\lambda r, y)$.
\begin{defn}\label{bHolder}
The $b$\emph{-Hölder space} $\calC^{0,\alpha}_b(X)$ consists of the functions $u: X\to \R$ such that
\[
\|u\|_{b;0,\alpha} := \sup_{j\geq 1}\|u_j\|_{\calC^{0,\alpha}(A_1)} < \infty,
\]
Each summand on the right is the ordinary Hölder norm for functions on the fixed annulus $A_1$.

We then also define
\begin{itemize}
	\item $\calC^{\ell,\alpha}_b(X):= \{u:X\to \R\mid (r\del_r)^j\del_y^\beta u \in \calC^{0,\alpha}_b(X) \text{ for all }j+|\beta|\leq \ell\}$;
	\item $r^\mu \calC^{\ell,\alpha}_b(X) := \{r^\mu v: v\in \calC^{\ell,\alpha}_b(X)\}$;
	\item $r^\mu \calC^{\ell,\alpha}_{b,D}(X)$ for the closed subspace of $r^\mu \calC^{\ell,\alpha}_b(X)$ of functions which vanish at $r=1$.
\end{itemize}
These definitions are stated for scalar functions, but extend immediately to sections of the normal bundle $NX$.   The obvious norms on these spaces 
are denoted by $\|\cdot \|_{b; \ell, \alpha}$ and $\|\cdot \|_{b; \ell, \alpha, \mu}$, respectively, with $\|u\|_{b;\ell,\alpha,\mu}:= \|r^{-\mu}u\|_{b;\ell,\alpha}$. 

Finally, as noted earlier, these definitions extend immediately to functions or sections defined on any compact manifold with conic singularities; the norms 
above are used in neighborhoods of each conic point and ordinary H\"older norms are used elsewhere on $X$.
\end{defn}

It is clear from this definition that 
\begin{equation}
L: r^{\mu} \calC^{2,\alpha}_b(X) \longrightarrow r^{\mu-2} \calC^{0,\alpha}_b(X)
\label{map1}
\end{equation}
is bounded for any $\mu \in \RR$.  When supplemented by elliptic boundary conditions at $\del X$, this map is Fredholm for most, but not all, values of $\mu$. 
The exceptional values are the set of \emph{indicial roots} 
(defined below), and for these weights, \eqref{map1} does not have closed range.  This is roughly analogous to a well-known phenomenon in classical elliptic 
PDE that, for example on a closed manifold, $\Delta: \calC^{k+2} \to \calC^k$ does not have closed range, but closed range can be recaptured by moving to 
nearby H\"older spaces.

Before proceeding, we note that \eqref{bHolder} concerns scales of finite regularity, but leaves open the question of what it means for a function
on a conic space to have infinite regularity.  There are, in fact, two separate notions, one stronger than the other, called \emph{conormality}
and \emph{polyhomogeneity}, respectively. We review them now since they will appear frequently in this paper. 

Before doing so, we note that it is more convenient to work in the setting of manifolds with boundary or corners.  A conic space $X$ can be resolved
into a manifold-with-boundary by blowing up its vertices. Locally this corresponds to replacing a neighborhood of the vertex in the unit cone over $\Sigma$, written $C(\Sigma)$, by
the cylinder $[0, 1) \times \Sigma$.  Later we also encounter manifolds with edges, e.g., singularities such as those appearing in $C(\Sigma) \times \RR^N$,
and these too can be resolved into manifolds-with-boundary by blowing up the vertex of the conic factor before taking the product. We will also
encounter manifolds with iterated conic singularities and will explain at that point how these more complicated singularities can be resolved
to produce manifolds with corners. 

\begin{defn}\label{def:polyhomog}
Let $S$ be a manifold-with-boundary, and suppose that $(x,y)$ is a local coordinate system near a boundary point where $x \geq 0$ is
a boundary defining function.   A function $u = u(x,y)$ is said to be \emph{conormal at $\del S$} if $|(x\del_x)^j \del_y^\beta u| \leq C_{j \beta}$ for
every $j$ and $\beta$.  (This condition can be replaced by one where all of these derivatives lie in any fixed function space, e.g., $L^2$ or
$x^\delta L^2$ for some fixed $\delta$; the crucial aspect is that of {\it stable regularity} -- it doesn't get worse when differentiated with these particular
vector fields.) Clearly any conormal function is smooth in the interior of $S$.

Next, a conormal function $u$ is said to have a classical (or {\it polyhomogeneous}) expansion 
\[
u \sim \sum_j \sum_{p=0}^{N_j}  a_{j p}(y) x^{\gamma_j} (\log x)^p,
\]
where $\gamma_j$ is a sequence of complex numbers with real parts tending to $+\infty$ and for each $j$ there are at most finitely many log factors,
if the difference between $u$ and any partial sum decays as $x\to 0$ like the order of next term in the expansion, with the same being true for any derivative. This
implies that all coefficient functions $a_{j p} (y)$ are $\calC^\infty$.  

As already noted, in our applications below, the manifold-with-boundary $S$ is obtained from a manifold with conic (or edge) singularities by blowing up these
conic vertices. 

In our geometric applications below, the logarithmic factors never appear in the leading coefficients, so we systematically omit mention of them
for notational simplicity. Their presence does not affect any of the arguments.

Finally, suppose that $S$ is a manifold with corners.  This means that any point $p \in S$ has a neighborhood which is diffeomorphic to a neighborhood
of $0$ in $(\RR^+)^\ell \times \RR^{n-\ell}$. Given such a point, let $(x_1, \ldots, x_\ell, y)$ be a set of local coordinates, with $p$ corresponding to $0$.
We say that $u$ is conormal in this neighborhood if it has stable regularity with respect to the vector fields $x_i \del_{x_i}$ and $\del_{y_j}$, and it is polyhomogeneous
if it has a joint expansion near this corner of the form
\[
u(x,y) \sim \sum_{\beta,\vec{p}} a_{\beta, \vec{p}}(y) x^\beta (\log x)^{\vec p},
\]
where $\beta$ and $\vec{p}$ are multi-indices, with each $p_j \in \mathbb N$, and $(\log x)^{\vec p} = (\log x_1)^{p_1} \ldots (\log x_\ell)^{p_\ell}$.  (This is
called a \emph{product type expansion} at the corner.) 

A comprehensive reference for polyhomogeneity can be found in \cite[Appendix A]{Maz}.
\label{polyhom}
\end{defn}

\medskip

Now let us return to the main theme of this subsection.  In the conic manifold $X$, restrict to a neighborhood $\calU_\ell$ of a conic point and
replace $X$ by the exact cone $C$. Denote by $\{\phi_j,\lambda_j\}_{j\geq 1}$ the eigendata for $L_\Sigma$, where $\Sigma$ is the link of
this cone, with the convention that $\lambda_j\nearrow +\infty$. This decomposition extends over the entire cone and defines a reduction of
$L$ to the ordinary differential operators $L_j = \del_r^2 + (k/r)\del_r - \lambda_j/r^2$. There are two corresponding independent solutions,
obtained from the ansatz $u = r^\gamma \phi_j$  by noting that
$$
L(r^\gamma\phi_j) = (L_\Sigma + \gamma^2 + k\gamma)r^{\gamma-2}\phi_j.
$$
This vanishes precisely when $\gamma^2+k\gamma - \lambda_j = 0$, i.e., 
$$
\gamma = \gamma_j^\pm := -\frac{k}{2}\pm\sqrt{\frac{k^2}{4}+\lambda_j}.
$$
Thus any eigenfunction $\phi_j$ extends to a two-dimensional space of solutions $u$ to $Lu=0$ on $C$. 

The values $\gamma_j^\pm$ are called the \emph{indicial roots} of $L$ at this conic point. The functions $\{r^{\gamma_j^\pm}\phi_j\}_{j=1}^\infty$ 
span the entire space of solutions to the homogeneous equation $Lu=0$ on $C$.  Notice that these indicial roots are symmetric around 
$-k/2$, and if $\lambda_j>-k^2/4$ then all the indicial roots are real. For simplicity, we assume that $\lambda_j>-k^2/4$ for all $j$, which
thus divides the set of indicial roots clearly into the subset of those which are greater than $-k/2$ and those which are less than $-k/2$.
It is useful when phrasing the mapping properties below to divide the indicial roots into two disjoint sets, and this division is clear in this case.
However, if $\lambda_j < -k^2/4$ for some $j$, solutions take the form $r^{-k/2 \pm i \sigma}$ for some real $\sigma$, and if $\lambda_j = -k^2/4$, 
the solutions are $r^{-k/2}$ and $r^{-k/2}\log r$. There are only a finite number of these exceptional cases, and in any one of these, we specify one
solution out of each of these two-dimensional families. This is a minor issue and we leave details (and all necessary small restatements of
results below) to the reader.

Here are the two main analytic results. 
\begin{prop}[{\cite[Theorem 4.26]{Maz}}]\label{prop:fredholm}
Given the conic manifold $X$, let $\Gamma$ denote the union of the sets of indicial roots $\{\gamma_j^\pm\}$ over all conic points of $X$.  For simplicity, assume
all $\gamma_j^+$ lie in $(-k/2, +\infty)$. Then for any $\mu \in \RR \setminus \Gamma$, the map
\begin{equation}
L: r^\mu \calC^{\ell+2,\alpha}_{b,D}(X;NX)\longrightarrow r^{\mu-2}\calC^{\ell,\alpha}_b(X;NX)
\label{bmapD}
\end{equation}
is Fredholm.  When $X$ is the exact cone $C$, this map is injective when $\mu>\gamma_1^-$ and surjective when $\mu<\gamma_1^+$, and 
hence invertible when $\mu \in (\gamma_1^-,\gamma_1^+)$.   In general, its nullspace and cokernel are constant as $\mu$ varies on any 
interval $(a,b)$ disjoint from $\Gamma$, and its index is monotone non-increasing, with jumps precisely at the indicial weights. This map
does not have closed range when $\mu \in \Gamma$.   
\end{prop}

\begin{prop}[{\cite[Corollary 4.19]{Maz}}]\label{prop:reg}
With all notation as above, if $u \in r^\mu \calC^{2,\alpha}_b$ and $Lu = 0$, or more generally if $Lu = f$ is supported away from the conic points 
then, near each $p_\ell$, $u$ admits a polyhomogeneous asymptotic expansion of the form
\[
u \sim \sum_{j} \sum_{m=0}^\infty  c_{j,m} r^{\gamma_j + m} \phi_j(y),
\]
where the $\gamma_j$ lie in the set of all indicial roots which are greater than $\mu$ (so in particular, at most a finite number of the
$\gamma_j^-$ appear), and the extra factors $r^m$ are included because for a metric which is only asymptotically conic, one requires extra correction 
terms to make this series solve the equation formally.  The coefficients $c_{j,m}$ are constants. 
\end{prop}

In our main geometric application, we must assume that the weight parameter satisfies $\mu > 1$. This is because we consider
normal graphs over $X$ by functions decaying like $r^\mu$ and this restriction ensures that the tangent cones of these normal graphs coincide
with the tangent cones of $X$ at these conic points.   Observe, however, that the value $1$ always appears in the set $\Gamma$. 
Indeed, the infinitesimal variations of one-parameter families of minimal solutions are solutions of the Jacobi equation, and the
`obvious' family of non-linear deformations obtained by rotating $C$ about its vertex gives rise to the indicial root $1$.

This highlights the basic tension in this problem.  It is necessary for geometric reasons to work on spaces with weights $\mu > 1$, but 
$L$ usually has non-trivial cokernel on such spaces.  In other words, the geometric deformation problem for embedded minimal submanifolds
is always obstructed.

Motivated by this and using the above two results, we prove the following.
\begin{prop}\label{nullspace}
Given any $\mu > 1$, there exists a subspace $W = W_\mu$ of finite codimension in $\calC^{2,\alpha}(\del X)$ such 
that there exists a solution to $Lu = 0$ on $X$ with $u \in r^\mu \calC^{2,\alpha}_b(X)$ and $u|_{\del X} = \phi$ if and only if $\phi \in W$. 
\end{prop}
\begin{proof}
Extend any $\phi \in \calC^{2,\alpha}(\del X)$ to a function $\tilde{\phi} \in \calC^{2,\alpha}_b(X)$ supported near $\del X$ and
which equals $\phi$ at the boundary. Write $f = L \tilde{\phi}$, and then select a value $\mu'$ so that \eqref{bmapD} is surjective.
Choosing a right inverse $G$ for this map, $w  = Gf \in r^{\mu'}\calC^{2,\alpha}_b$ with $Lw = f$, then by Proposition~\ref{prop:reg}, 
\[
w \sim \sum_{\gamma_j > \mu'}  c_{j,m} r^{\gamma_j + m}\phi_j(y).
\]
Each of the boundary trace maps $T_{j,m}: f \mapsto w = Gf \mapsto c_{j,m}$ is continuous, and $W$ is the intersection of the kernels of
$T_{j,m}$ over all such pairs with $\gamma_j + m < \mu$. 
\end{proof}
Smale \cite{Smale93} proves this slightly differently using Green's theorem; that method identifies $W$ as the annihilator
of the set of normal derivatives of elements of the kernel of $L$ in $r^{\mu'}\calC^{2,\alpha}_b$.

When $X$ is an exact cone, there is an explicit identification of $W$. Namely, for each $j$, in order for at least one of the two solutions
$r^{\gamma_j^\pm} \phi_j(y)$ to lie in $r^\mu \calC^{2,\alpha}_b(C)$, we need $\gamma_j^+ \geq \mu$.
Thus if we choose $J$ to be the unique integer such that $\gamma_J^+ < \mu < \gamma_{J+1}^+$, and if $\Pi_J$ is the finite rank $L^2$ orthogonal
projection of $\calC^{2,\alpha}(\Sigma)$ onto the span of the $\phi_j$ with $j \leq J$, then $W = (I - \Pi_J) \calC^{2,\alpha}(\Sigma)$.
In general, we write $\Pi_W$ for the $L^2$ orthogonal projector, which maps $\calC^{2,\alpha}(\Sigma)$ onto $W$. 

In any case, given $\mu > 1$, we can define the Poisson operator
\begin{equation}
\mathcal{P}: W_\mu \longrightarrow r^\mu \calC^{2,\alpha}_b(X) \cap \ker L
\label{Poissonopb}
\end{equation}
by setting $u = \calP(\phi)$ to be the unique solution of $Lu = 0$ with $u|_{\del X} = \phi$. 

We need one further preparatory result.
\begin{prop}\label{compker}
The (infinite dimensional) kernel of $L$ in $ r^\mu \calC^{2,\alpha}_b(X; NX)$ is complemented; in other words, there exists a closed subspace $E$ such that
\[
r^\mu \calC^{2,\alpha}_b(X; NX) = \ker L \cap r^\mu \calC^{2,\alpha}_b(X; NX) \oplus E.
\]
\end{prop}

The abbreviated proof of this result is that if $\Pi_{\mathrm{Ker}}$ denotes the orthogonal projection to this kernel in the weighted
$L^2$ space $r^{\mu - (k-1)/2} L^2$, then it can be shown that this same operator (or rather, its Schwartz kernel) defines a bounded
map on $r^\mu \calC^{2,\alpha}_b$, and we then define $E$ to be the image of $I - \Pi_{\mathrm{Ker}}$ in this weighted H\"older space.

We make one small emendation of this result, by noting that we can choose $E$ so that if $u \in E$, then $\Pi_W (u|_{r=1}) = 0$.
Indeed, there is a continuous restriction map $\mathcal R$ from $E$ to $\calC^{2,\alpha}(\del X)$, and we can then replace $E$ by
\[
(I - \calP \circ \Pi_W \circ \mathcal R) (E).
\]

To explain this more carefully, we take a small detour. One may work with weighted $b$\emph{-Sobolev spaces} $r^\nu H^k_b$ instead
of weighted $b$-H\"older spaces. The definitions of these are what one would expect, namely $u \in H^k_b$ if $(r\del_r)^j \del_y^\beta u \in 
L^2$ for $j + |\beta| \leq k$. Although Sobolev regularity is not strictly comparable to H\"older regularity, we say that $r^\mu \calC^{\ell_1,\alpha}_b$ 
(for any $\ell_1, \alpha$) is {\it commensurate with} $r^{\nu} H^{\ell_2}_b$ (again for any other $\ell_2$) if $\mu = \nu + (k-1)/2$.  
This is motivated by the fact that $r^\mu \in r^\nu L^2$ locally near $r=0$ if and only if $\mu > \nu + (k-1)/2$.  All of the analytic results 
above are obtained using the theory of $b$-pseudodifferential
operators. We say more about all this below, but for the moment we record a basic fact of this theory:  a $b$-operator is bounded on a given weighted 
H\"older space $r^\mu \calC^{k,\alpha}_b$ (with appropriate gain or loss of $b$-regularity) if and only if it is also bounded on the commensurate 
weighted Sobolev space $r^{\nu} H^k_b$, $\nu = \mu - (k-1)/2$.  This is not a difficult fact, and follows from the description of 
$b$-pseudodifferential operators in terms of the pointwise behavior of their distributional Schwartz kernels.  

This observation about the boundedness of $b$-pseudodifferential operators has another important consequence.  Suppose we have
proved that $L$ is injective on $r^\mu \calC^{2,\alpha}_b$ for some non-indicial weight $\mu$. A duality argument can be used to prove that 
$L: r^{\mu'}\calC^{2,\alpha}_b \to r^{\mu'-2}\calC^{0,\alpha}_b$ is surjective, where $\mu'$ is a weight {\it dual} to $\mu$. 
Of course, duality is not meant literally since we are using H\"older spaces, but we argue as follows. The injectivity of $L$ on the first space 
implies that there exists an operator $G: r^{\mu-2}\calC^{0,\alpha}_b \to r^\mu \calC^{2,\alpha}_b$ such that $G \circ L = I$.  A priori,
this operator $G$ is only defined abstractly, by functional analysis, but it can be shown that it is actually a $b$-pseudodifferential operator. 
Using the `transferability of boundedness'  explained above, we obtain  that $G: r^{\nu-2} L^2 \to r^{\nu}H^2_b$ for $\nu = \mu - (k-1)/2$,
so in particular, $G \circ L = I$ still holds, first as a statement about distributional Schwartz kernels, but then in an operator-theoretic sense. 
In other words, $L$ is injective on $r^\nu H^2_b$. 

Now, if $u \in r^\nu H^2_b$ and $v \in r^{\nu'} H^2_b$, then both sides of the unweighted $L^2$ pairing
\[
\langle Lu, v \rangle = \langle u, Lv \rangle
\]
make sense precisely when $\nu' = - (\nu-2)$. Thus $\nu' = 2-\nu$ is the dual $L^2$ weight, and reinserting the shift above,
the H\"older weight $\nu' = \mu' - (k-1)/2 = 2 - (\mu - (k-1)/2)$, or
\[
\mu' = k+1 - \mu
\]
is the dual H\"older weight.  By this surjectivity, there is an operator $G^t: r^{\nu'-2}L^2 \to r^{\nu'}H^2_b$, and using all the remarks
above, this same operator defines a bounded map $r^{\mu'-2} \calC^{0,\alpha}_b \to r^{\mu'} \calC^{2,\alpha}_b$ which satisfies 
$L G^t = I$.  This shows that $L$ is surjective on the dual weighted space as claimed! 

We also have the following crucial surjectivity statement.
\begin{prop}
If $\mu$ is any non-indicial weight, then the mapping
\begin{equation}
L:  r^{\mu} \calC^{2,\alpha}_b(X; NX) \longrightarrow r^{\mu-2} \calC^{0,\alpha}_b(X; NX)
\label{noD}
\end{equation}
is surjective.
\label{nonDbsurj}
\end{prop}
The difference with previous results is that we are not imposing homogeneous Dirichlet conditions at $\del X$.
\begin{proof}
By  the same line of reasoning as above, \eqref{noD} is surjective if and only if its dual mapping
\[
L: r^{\mu'} \calC^{2,\alpha}_{b,C}(X; NX) \longrightarrow r^{\mu'-2} \calC^{0,\alpha}_b(X;NX)
\]
is injective. The subscript $C$ on this domain space denotes the space of H\"older functions for which {\it both} the Dirichlet
and Neumann conditions are satisfied.   However, it follows from the classical uniqueness of the Cauchy problem for
second order elliptic equations that there are no non-trivial solutions to this dual problem.
\end{proof}

Since we shall be invoking properties of various classes of degenerate pseudodifferential operators below, let us interrupt the main
flow to say a few words about the class of $b$-pseudodifferential operators, refering to \cite{Maz} for complete details.   In general, pseudodifferential
operators are a special class of integral operators $u \mapsto \int G(x,y) u(y)\, dy$ where $G$ is a distributional kernel with special 
properties. On a closed manifold, this kernel is smooth away from the diagonal and required to have a `classical' singularity (modelled
on that of the Newtonian potential, $|x-y|^{2-n}$, but possibly with different homogeneities) along the diagonal.  There is then a complete
theory about how these operators may be composed and the precise functions spaces on which they provide bounded maps. 
On a manifold-with-boundary $Y$ (or more generally, a manifold with corners), this type of description is inadequate in that it does not allow
for non-smooth behavior on the boundary faces $\del Y \times Y$ and $Y \times \del Y$. In particular, if $G$ is a parametrix for an 
elliptic conic operator, for example, then one expects some sort of exceptional singular behavior at $\del Y \times \del Y$.  The geometric
microlocal method, pioneered by Melrose, gives a way to explain these extra singularities, by requiring $G$ to not only have its classical diagonal singularity,
but to also have classical expansions at the side faces and also at a new front face obtained by a (real) blow-up of the corner. 
This is the basic ansatz for $b$-pseudodifferential operators (other degeneracy classes discussed below involve more intricate
blow-ups).  The point finally is that with this ansatz there exists a sufficiently broad class of integral kernels, i.e., pseudodifferential
operators, so as to contain good approximate inverses for conic elliptic operators.   Furthermore, the recognition of a parametrix (or
generalized inverse) as lying in this class entails knowing precise pointwise behavior of the Schwartz kernel, which may then
be parlayed into sharp mapping and regularity properties. 

We can now state the main non-linear perturbation result of \cite{CHS} and \cite{Smale93}:
\begin{theorem}\label{CHS}
Fix a non-indicial value of $\mu>1$, and let $W$ be the subspace of $\calC^{2,\alpha}(\del X)$ consisting of boundary values of solutions to $Lu = 0$
with $u \in r^{\mu} \calC^{2,\alpha}_b(X)$. Then there exists an $\eps  > 0$ sufficiently small (depending on the various geometric and analytic 
quantities above) such that if $\phi \in W$ and $\|\phi\|_{2,\alpha} < \eps$, then there exists $u\in r^\mu \calC^{2,\alpha}_b(X;NX)$ 
such that
$$
\mathcal{H}(u) = 0 \qquad \text{and} \qquad \Pi_W (u|_{r=1}) = \phi.
$$
Here $\mathcal H$ is the minimal surface operator acting on sections of the normal bundle of $X$.
\end{theorem}

\begin{proof}
Given all of the linear tools introduced above, this amounts to observing the following. First, $\calH$ defines a smooth
map from a neighborhood of $0$ in $r^{\mu}\calC^{2,\alpha}_b$ to $r^{\mu-2}\calC^{0,\alpha}_b$.  The only point to note for this
is that in the non-linearities of $\mathcal H$, there are no quadratic expressions in second derivatives. 

The linearization of $\mathcal H$ at $u=0$ is the Jacobi operator $L$ and, by Proposition~\ref{nonDbsurj}, this linearization 
is surjective. Proposition \ref{compker} shows that  its kernel has a topological complement, hence the implicit function theorem
gives a smooth map $S$ from a neighborhood of $0$ in $\mathrm{ker} L \cap r^{\mu} \calC^{2,\alpha}_b$ to its complement $E$ such
that $\mathcal H( w + S(w)) = 0$, and all solutions near to $0$ are of this form.  We recall finally, from Proposition \ref{nullspace}, 
that elements of this nullspace are in bijective correspondence with elements of $W$.
\end{proof}

\subsection{Eigenfunctions on $S^2$ with branch points}\label{sec:harmonic} 
A minimal deformation of a multiplicity two submanifold may of course consist of two distinct minimal submanifolds, but our interest here is in constructing  
deformations which are branched. The infinitesimal version of a (suitably nice) branched minimal surface over a disk is a multi-valued harmonic function which vanishes  
to a certain order along its branching set. In this subsection, we review some facts about two-valued harmonic functions, and introduce the 
main examples used later.

In the present context there are two different ways to define two-valued harmonic functions. The first, systematically developed by Almgren \cite{Alm}, uses functions taking values 
in $\calA_2(\R^k)$, the set of sums of two Dirac masses in $\R^k$. This definition works well in low regularity situations where less structure is known 
a priori. The second, which is the one we use here as we have better a priori regularity, is to regard such functions as sections of a certain flat rank $1$ real line bundle away from a 
(codimension $2$) branching locus (cf.~\cite{Don, TW20}).

In general the problem may be phrased as follows. Given a Riemannian manifold $(M^n,g)$ and a closed subset $\Sigma$ of codimension $2$, we are 
interested in finding multi-valued harmonic functions on $M\setminus \Sigma$ which branch along $\Sigma$. For simplicity of exposition, assume 
that $\Sigma$ is either a smooth $(n-2)$-dimensional submanifold or, at worst, a smoothly stratified subset of $M$ with the principle stratum being of 
dimension $n-2$. (In applications where these multi-valued functions arise, see e.g.\ \cite{Tau, Zha}  for example, it is only known that $\Sigma$ is 
countably $(n-2)$-rectifiable. However, we are not concerned here with this most general setting.)

Fix a flat real metric line bundle $\calI$ over 
$M\setminus \Si$ with monodromy $-1$ on small loops around $\Sigma$. In other words, a section of $\calI$ is locally represented by a real-valued 
function, but parallel transport of this function as a section of $\calI$ on a loop encircling $\Sigma$ once acts by multiplicity by $-1$. We can 
thus regard any section of $\calI$ as a two-valued scalar function where the value of $u$ at any point $p\not\in\Sigma$ is indistinguishable 
from $-u(p)$. The prototypical example is the function $u(z) = \mathrm{Re}\,\sqrt{z}$ on $\CC\cong\R^2$. We call $\Sigma$ the \emph{branching set} of $u$.  
For readers less familiar with this bundle formulation, a simple example to gain intuition is provided by the non-orientable $\RR$-bundle over $S^1$ 
(i.e., the M\"obius strip).  A section of this corresponds to a function $u$ on $\RR$ which transforms by $u(x + 2\pi) = -u(x)$. Thus the value 
$u(x)$ at a point $x$ is indistinguishable from its value $-u(x)$ at the equivalent point $x+2\pi$ in $\RR/2\pi \ZZ$.  
This is equivalent to Almgren's formulation where $u$ is a map $M\to \calA_2^0(\R)$ into the set of sums of two Dirac masses 
in $\R$ with average zero.  Donaldson \cite{Don} gives a slightly more general definition  where $\calI$ is a flat affine bundle; in the present 
setting, where the only monodromy comes from the $-1$ action around $\Sigma$, this is equivalent to Almgren's definition.

Since $\calI$ is flat, the notion of a harmonic section is well-defined (and independent of the choice of $\calI$), and we call such an object a
\emph{$\ZZ_2$-harmonic function} or \emph{two-valued harmonic function}. We will impose certain decay restrictions on $u$ at $\Sigma$ below. 

There are natural orders of blow up or decay at $\Sigma$ for solutions of this problem. Indeed, when $\Sigma$ is a \emph{smooth} codimension $2$ submanifold, 
the frequency of a $H^1_{\text{loc}}$ $\Z_2$-harmonic function $u$ at points in $\Sigma$ takes values in $\{\tfrac{1}{2},\tfrac{3}{2},\tfrac{5}{2},
\dotsc\}$. This follows from a stronger and more general statement, namely in cylindrical Fermi coordinates $(r,\theta,y)$ around $\Sigma$ 
(here, $y\in \Sigma$), any $L^2_{\text{loc}}$ $\Z_2$-harmonic function $u$ has an asymptotic expansion of the form 
\begin{equation}
\begin{split}
 & u(r,\theta, y) \sim \,  r^{-1/2}( a_0(y) \cos(\theta/2) + b_0(y) \sin(\theta/2))\\
 & \hspace{10em} + r^{1/2} (a_1(y) \cos(\theta/2) + b_1(y) \sin(\theta/2))\\ 
 & \hspace{15em}  + r^{3/2} (a_2(y) \cos (3\theta/2) + b_2(y) \sin( 3\theta/2)) + \cdots
\end{split}
\label{asymp}
\end{equation}
This is explained more carefully in \cite{HMT23}, but see \cite[Section 7]{Maz} for a proof. The individual terms correspond to, of course, the functions
$(\sqrt{z})^{2k+1}$ and $(\sqrt{\bar{z}})^{2k+1}$, $k\in \{-1,0,1,\dotsc\}$, where $z$ is a complex coordinate in the normal variables in
a Fermi coordinate system around $\Sigma$. 

As a matter of general background, we recall from \cite{HMT23, Maz} that there is an infinite dimensional space of $\ZZ_2$-harmonic functions 
which blow up like $r^{-1/2}$ at $\Sigma$. When $M$ is a closed manifold, these are in bijective correspondence with pairs of leading coefficients $(a_0,b_0) 
\in H^{-1/2}(\Sigma) \oplus H^{-1/2}(\Sigma)$. Finding a two-valued harmonic function with specified leading coefficients $(a_0,b_0)$ is tantamount 
to solving an inhomogeneous Dirichlet problem in this setting, with $(a_0, b_0)$ serving as the Dirichlet data. By the maximum principle, there are no (non-zero)
two-valued solutions of $\Delta u = 0$ which have both $a_0 = b_0 = 0$.  We also remark that this expansion only holds in
a weak sense if $a_0, b_0 \not\in \calC^\infty$, but we refer to \cite{Maz} for a careful description of that phenomenon since it does not 
arise in our particular situation.  In the situations considered here, $a_0, b_0 \in \calC^\infty$ and the expansions are polyhomogeneous 
as in Definition~\ref{polyhom}.

If $M$ is a manifold-with-boundary, e.g., $M = B^3$ (as in our setting), and $\Sigma$ is transverse to the boundary, then there are infinitely many 
$\ZZ_2$-harmonic functions which vanish like $r^{1/2}$ on $\Sigma$. These are determined by specifying  boundary values at $\del M$. Note
that these boundary values are themselves $\ZZ_2$-valued, and branch on $\Sigma \cap \del M$. 

The $\ZZ_2$-harmonic functions which are important in geometric applications, including the one here, are not only bounded, i.e.,
omit the $r^{-1/2}$ term in their expansion, but vanish to order $3/2$ or faster. In other words, they also omit the $r^{1/2}$ terms,
and hence the coefficients $a_0, b_0, a_1, b_1$ all vanish.   
This extra requirement is tantamount to imposing that the \emph{entire} Cauchy data of $u$ on $\Sigma$ vanishes, which then
makes this a sort of free-boundary problem.  As such, it is highly ill-posed, and one should only expect solutions to exist in exceptional 
circumstances. In particular, the set of pairs $(M,\Sigma)$ (and underlying metric $g$ on $M$) for which such a non-trivial solution is supported
should be highly constrained. This instability is prominent in Donaldson's investigation \cite{Don} of the deformation theory of these 
special $\Z_2$-harmonic functions.

We henceforth restrict attention to a particular subclass of functions on the ball obtained by imposing a discrete symmetry; this turns
out to create a significant simplification.  The relevance of this restriction is a central feature in the paper of Taubes--Wu \cite{TW20}. 
Their interest in that paper is in constructing $\ZZ_2$-harmonic functions decaying at order $3/2$ or greater along branching sets which
are more complicated than simple curves. They observe that for a small list of suitable group actions, the $\ZZ_2$-harmonic functions which are 
in $L^2$ and equivariant with respect to one of these actions necessarily satisfy this faster decay at the branching set, which
is equal to the set of points where the isotropy subgroup is non-trivial.

The simplest example of their construction is described as follows. Let $\Gamma$ be the group of all (not necessarily orientation-preserving) 
symmetries of the regular tetrahedron in $\R^3$ (this can be identified with $S_4$, the group of symmetries on four letters by permuting 
the vertices $p_1,\dotsc,p_4$).  `Inflating' the regular tetrahedron into a spherical tesselation by equilateral triangles $\{T_\ell\}_{\ell=1}^4$,
then $\Gamma$ is generated by reflections across the edges of these triangles. (One can thereby alternately think of $\Gamma$ as acting by rigid motions
on the sphere or on $\RR^3$.) The isotropy group of this action on $\R^3$ is non-trivial on each of the four rays $R_j = \{rp_j:r\geq 0\}$, 
and the $\ZZ_2$-harmonic functions produced by Taubes--Wu have decay rate $3/2$ on the branching set $\mathcal{R} := R_1\cup\cdots\cup R_4$. 

Roughly speaking, their method is to first find $\Z_2$-harmonic functions which branch along $\mathcal{R}$, but only decay to order $1/2$ there, and 
then use the group symmetry to improve the decay rate to order $3/2$ or greater. 

In detail, they seek harmonic functions which are homogeneous on $\R^3$, i.e., have the form $u(r,y) = r^\gamma\psi(y)$, where $r\geq 0$ and 
$y\in S^2$ are spherical coordinates. Here $\psi$ is itself a $\Z_2$-valued function on $S^2$ which branches at the points $p_j$. As in the 
single-valued case, an easy calculation shows that $\Delta u = 0$ is equivalent to the eigenfunction equation
\begin{equation}\label{quad}
\Delta_{S^2} \psi  + \gamma(\gamma-1) \psi = 0.
\end{equation}
This reduces the problem to seeking $\Z_2$-eigenfunctions of $\Delta_{S^2}$ which branch at the $p_j$; a subsequent argument then shows that these 
in fact vanish at a faster rate at each such point. 

We briefly pause to note that there is a slightly more general formulation where the first part of this may be carried out. Namely, fix an even number 
of points $p_1,\dotsc,p_{2k}$ in $S^2$ (there must be an even number to support a line bundle with the correct holonomy) and search for $\Z_2$-eigenfunctions 
of $\Delta_{S^2}$ which decay to order $1/2$ and branch at exactly these points. 
This can be done abstractly by considering the Friedrichs extension of the Laplacian (acting on sections of $\calI$), which is obtained from the clossure
of $\calC^\infty_0( S^2 \setminus \{p_1, \ldots, p_{2k}\})$ using the quadratic form $Q(u) = \int |\nabla u|^2$. This general process produces a
self-adjoint operator with discrete spectrum, and each of its eigenfunctions branches and decays to order $1/2$ at every one of the $p_j$ (this is 
built into the Friedrichs construction).  However, it is still quite rare to find configurations of points where at least some of these eigenfunctions
vanish to order $3/2$ at least at all of the $p_j$. We refer to \cite{He,TW24,FS25} for more on this.

There are two equivalent ways to think about the $\Gamma$-equivariance: 
\begin{enumerate}
\item [(1)] In the first, again let $\mathcal{T} = \{T_\ell\}_{\ell=1}^4$ be the four spherical triangles (each of which is a fundamental domain for the action 
of $\Gamma$ on $S^2$).  Let $\psi_1$ be the ground state eigenfunction of the restriction of $\Delta_{S^2}$ to $T_1$ with Dirichlet conditions 
at $\del T_1$. Now extend $\psi_1$ to the rest of $S^2$ by successive reflection across the sides of $T_1$, each time defining the extension by 
odd reflection. The Schwarz reflection principle ensures that this extended function is smooth away from the $p_j$. Since each $p_j$ is a vertex of 
exactly three triangles, this continued reflection produces a function which only ``closes up'' after two rotations around each $p_j$. This extended
eigenfunction is a section of $\calI$, and since it is bounded, it must lie in the Friedrichs domain.
\item [(2)] The second is to consider the lift of the action of $\Delta_{S^2}$ to the branched double-cover $\widetilde{M} \to S^2$, where the
branching set consists of the four points $p_j$. This cover is a torus with four conic points, each with cone angle $4\pi$. We may consider the
Friedrichs domain of the Laplacian on this conic surface, and then restrict to the subspace of functions which transform by $-1$ when reflected 
across the edges of each of the lifts of the triangles $T_\ell$. This is the same as restricting to the subspace of functions which transform 
according to the representation $\Gamma\to \Z_2$ which sends $\gamma\mapsto \text{sign}(\gamma)$ to the sign of the permutation in 
$S_4$ corresponding to $\Gamma$.
\end{enumerate}
In either formulation, one obtains an infinite collection of $\Z_2$-valued eigenfunctions $\{\phi_\ell\}_{\ell=1}^\infty$ on $S^2$ which branch at the $p_j$ and are bounded.
Any such $\phi_\ell$ has an asymptotic expansion near each $p_j$ of the same form as in \eqref{asymp},  but where the distance function $r$ to $\Sigma$
is now replaced by the spherical distance $\rho$ to $p_j$ (this is defined and smooth in some small disk around each $p_j$).  Since $\phi_\ell$ is bounded,
its expansion can contain only terms of the form $\rho^{ k+\frac12}$ with $k \geq 0$.  However, the coefficients $e^{\pm i \theta/2}$ of $\rho^{1/2}$
do not transform correctly with respect to the action of the subgroup of $\Gamma$ which fixes $p_j$. Namely, since $u$ must transform by $-1$ 
under $\theta \mapsto \theta + 2\pi/3$, the same must be true for the coefficient of each $\rho^{k+1/2}$. Hence the coefficients $a_1, b_1$ vanish.  
Notice that the coefficients $e^{\pm 3i\theta/2}$ of $\rho^{3/2}$ do transform correctly under this action, hence are not excluded.  This proves the following:
\begin{lemma}
Denote by $L^2_\Gamma(S^2, \calI)$ the subspace of $L^2$ sections of $\calI$ which satisfy the above $\Gamma$-equivariance property. Then each
eigenfunction $\phi_\ell\in L^2_\Gamma(S^2;\calI)$ of $\Delta_{S^2}$ has a polyhomogeneous expansion at each $p_j$ with leading term 
$\rho^{3/2}(a_2\cos(3\theta/2) + b_2\sin(3\theta/2))$.  In fact, choosing the angular coordinate so that one edge of the triangle $T_1$ 
corresponds to $\theta = 0$, the coefficient of $\rho^{3/2}$ is simply $b_2 \, \rho^{3/2} \sin (3\theta/2)$. 
\end{lemma}

\begin{defn} To simplify notation below, for any of the function spaces considered here consisting of sections of the line bundle $\calI$ with 
various regularities, we add a subscript $\Gamma$ to indicate the subspace consisting of elements which transform by $-1$ with respect
to the reflections in $\Gamma$.  Thus we shall work entirely within the subclass of $\Gamma$-equivariant sections. Notice that any one
of these sections necessarily vanishes on the boundaries of each of the spherical triangles $T_j$ (or $C(T_j)$ in $B^3$). 
\end{defn}

We next examine the behaviour at $0$ of any $\Z_2$-harmonic function obtained this way. Recall the initial ansatz that $u = r^\gamma \psi(y)$ where $\psi$ is 
a $\ZZ_2$-eigenfunction of $\Delta_{S^2}$.  Following \eqref{quad}, the exponents are $\gamma_\ell^\pm = -\frac{1}{2}\pm \sqrt{\frac{1}{4}+\lambda_\ell}$.  
Our solutions must be bounded, so the roots $\gamma_\ell^-$ are discarded. In fact, we are only interested in solutions which vanish to order greater than $1$ 
at $0$ since in the corresponding minimal surface problem, this corresponds to solutions where the tangent cone of the graph is the original cone. To that end, 
observe that
$$
\gamma_\ell^+ > 1 \quad \Longleftrightarrow \quad \lambda_\ell>2.
$$
The beautiful fact is that this is no restriction at all! Each spherical triangles $\mathcal{T}_i$ lies in a hemisphere and each eigenfunction $\phi_\ell$ 
vanishes on $\del T_j$. The lowest Dirichlet eigenvalue of this hemisphere is $2$, hence by domain monotonicity of Dirichlet eigenvalues, the condition 
$\lambda_\ell>2$ holds for every $\ell$. This is the key to our main non-linear perturbation result in Section \ref{sec:non-linear-perturbation}.

We note for later purposes that while the lowest eigenvalue of this equilateral spherical triangle is not known explicitly, numerical computations indicate that 
it is approximately equal to $5.159$, cf.\ \cite{RatTre}.  From this it follows that 
$$
\gamma_\ell^+>3/2 \qquad \text{for all }\ell.
$$
One consequence is that in these examples the frequency function is discontinuous at $0$.

Using all of this, we can now define the Poisson map action on sections of $\calC^{2,\alpha}_{b,\Gamma}(S^2;\calI)$ by
$$
\mathcal{P}(\phi) := \sum^\infty_{\ell=1}a_\ell r^{\gamma_\ell^+}\phi_\ell
$$
where $\phi = \sum_\ell a_\ell\phi_\ell$. It is straightforward to check using a barrier argument that for $\mu\in (1,\gamma_1^+)$ we have
$$
\mathcal{P}:\calC^{2,\alpha}_{b,\Gamma} (S^2;\calI)\longrightarrow r^\mu \calC^0_\Gamma(B^3).
$$
There is a much sharper regularity statement for solutions which we defer to Proposition \ref{poissonedge} and Proposition \ref{polyprop} below. A 
weaker statement following from the results above is the following:
\begin{prop}
Each $\Gamma$-equivariant $\Z_2$-harmonic function $u$ obtained as above satisfies
\begin{enumerate}
\item [(i)] $|u(r,y)|\leq C\dist(ry,\mathcal{R})^{3/2}$ near the (branching) rays $R_j$;
\item [(ii)] $|u(r,y)|\leq Cr^{3/2+\eps}$ as $r\to 0$, for some $\eps>0$.
\end{enumerate}
\end{prop}

To conclude, although we have restricted to the particular case of the tetrahedral tiling of $S^2$. Taubes--Wu discuss a handful of other examples, 
where the discrete group $\Gamma$ is replaced by a few other possible examples. For each of these the overall picture remains the same. 
This entire discussion can also be generalized to certain regular tilings of $S^{n-1}$. To keep the focus of this section simpler, we have restricted 
to these examples on $S^2$ and $B^3$, and defer discussion of the higher dimensional cases to Section \ref{sec:hypersurface}.

\subsection{An extension to cylinders}

There is one fairly easy extension of the results in Section \ref{sec:harmonic} to higher dimension, which we detail here. The discussion in 
Section \ref{sec:harmonic} gives two-valued harmonic functions on $B^3$ satisfying the extra vanishing condition at the branch set. In fact, this 
immediately generalizes to define two-valued harmonic functions on products of the form: $B^3\times Y$, where $Y$ is any compact manifold; 
or $B^3\times\R^N$ for some $N\geq 1$. In these cases, the group acts just on the first factor. We take an intermediate route here and consider 
the problem on $B^3\times\bbT^N$, where $\bbT^N$ is the standard (flat) $N$-torus.  This immediately yields results about two-valued harmonic 
functions on $B^3\times\R^N$ which are periodic in the $\RR^N$ factor.  It seems quite likely that one can take limits of sequences of such 
solutions, allowing the generators of the lattice defining the torus to tend to infinity, to obtain non-periodic solutions on $B^3 \times \RR^N$
which have some specified decay rate at infinity. However, we leave any such potential generalization to the reader. 

Recall that $p_1,\dotsc,p_4\in S^2$ are the vertices of the regular tetrahedron with corresponding rays $R_1,\dotsc,R_4$ and spherical triangles 
$T_1,\dotsc,T_4\subset S^2$. We write $(r,y)$ for the standard polar coordinates on $B^3$, with $y\in S^2$, and $z$ for the coordinates on 
$\bbT^N$. The extension of the results in Section \ref{sec:harmonic} to this setting can be stated as follows.

\begin{prop}\label{poissonedge}
Fix $\mu \in (1, \gamma_0^+)$. Let $\phi\in \calC^{2,\alpha}_{e,\Gamma}(S^2\times\bbT^N;\calI)$ be any two-valued function which branches along 
$\cup_{j}(p_j\times\bbT^N)$ and is odd with respect to reflections across any face $(\del T_j)\times\bbT^N$ in $S^2\times \bbT^N$.  Then there exists 
a harmonic function $u\in r^\mu \calC^{2,\alpha}_{\ie, \Gamma}(B^3\times\bbT^N;\calI)$ which satisfies $u|_{S^2\times \bbT^N} = \phi$ and branches 
along $\cup_{j}(R_j\times \bbT^N)$.
\end{prop}

The function space $\calC^{2,\alpha}_{\ie}(B^3\times\bbT^N;\calI)$ is defined in Definition \ref{ieHolder} below, and $\calC^{2,\alpha}_{\ie,\Gamma}$
consists of the $\Gamma$-equivariant elements, as usual. Similarly, see Definition \ref{eHolder} for $\calC^{2,\alpha}_{e}$.

Proposition \ref{poissonedge} can be proved using $L^2$ techniques, using the eigendecomposition $\{\phi_\ell(y)\}_{\ell=1}^\infty$ for 
$\Delta_{S^2}$ (with corresponding eigenvalues $\{\lambda_\ell\}_\ell$) acting on two-valued functions on $S^2$ to reduce to solving a sequence of 
equations on each eigenmode. Indeed,  
decomposing $u = \sum_\ell u_\ell(r,z)\phi_\ell(y)$ and $\phi = \sum_\ell\psi_\ell(z)\phi_\ell(y)$, then we reduce the problem to solving 
$$
L_\ell u_\ell := \left(\del_r^2 + \frac{2}{r}\del_r - \frac{\lambda_\ell}{r^2}+\Delta_{\bbT^N}\right)u_\ell(r,z) = 0 \qquad \text{with} 
\qquad u_\ell(1,z) = \psi_\ell(z)\in \calC^{2,\alpha}(\bbT^N).
$$
To solve this, there is an integral formula expressing $u_\ell$ in terms of $\psi_\ell$, and subsequently $\sum_\ell u_\ell\phi_\ell$ can be estimated in 
$L^\infty$ using appropriate barrier functions, giving the claimed rate of decay at $0$. This procedure can be viewed as giving the solution via an 
appropriate Poisson extension $u = \mathcal{P}(\phi)$. Alternatively, one can appeal to the systematic theory from \cite{Maz} (and further references 
cited below), which leads to significantly sharper results.

Next, in order to develop the perturbation theory for our setting which mirrors that we used in our proof of Theorem \ref{CHS} (the result 
of \cite{CHS}), we need to develop the corresponding Fredholm theory. In particular, we need to define the basic Hölder spaces on which 
$\Delta_{B^3\times\bbT^N}$ acts, and then we can state the main mapping and regularity results on these spaces. For these we follow \cite{Maz}, 
although in the interests of brevity we do not describe the pseudodifferential operator constructions used to prove these results. 
For simplicity, we refer to this operator simply as $\Delta$ below.

Much as in Definition \ref{bHolder}, the degeneracy structure of the operator serves as a 
guide in these definitions. We do this in two stages. In the first, we consider the simpler edge degeneracy exhibited by $\Delta$ near the edges 
$R_i\times\bbT^N$, but away from $\{0\}\times\bbT^N$, and in the second we consider the `deeper' degeneracy at this inner layer. In both these 
cases we use the expression 
$$
\Delta = \del_r^2 + \frac{2}{r}\del_r + \frac{1}{r^2}\left(\del_\rho^2+\frac{\cos(\rho)}{\sin(\rho)}\del_\rho + \frac{1}{\sin^2(\rho)}\del_\theta^2\right) + \Delta_z
$$
where we have written the Laplacian in terms of the polar variable $r\in [0,1]$ and polar coordinates $(\rho,\theta)$ on $S^2$ near any of the points $p_j$.
(The reader should keep in mind that because of the monodromy around each of the points $p_j$ on $S^2$, there is a genuine singularity at $\rho=0$.)

Thus for the first stage, consider the restriction of $\Delta$ to $\{c<r<1\}$ for some $c>0$. Observe that in this region,
\begin{equation}\label{1st}
\sin^2(\rho)\Delta = \sin^2(\rho)\left(\del_r^2+\frac{2}{r}\del_r+\Delta_z\right) + \frac{1}{r^2}\Big(\sin^2(\rho)\del_\rho^2 + 
\sin(\rho)\cos(\rho)\del_\rho + \del_\theta^2\Big),
\end{equation}
which is an elliptic combination of the vector fields $\rho\del_\rho$, $\rho\del_r$, $\rho\del_{z_i}$, and $\del_\theta$, with coefficients which are smooth 
for $\rho\geq 0$ (note that here we are using that $r>c$). The goal again is to define function spaces which respect the symmetries of these vector fields. 
Observe that since we are bounding $r$ away from $0$, these operators are locally invariant with respect to the dilation $(r,\rho,\theta,z)\mapsto 
(\lambda r, \lambda\rho,\theta,\lambda z)$ and translations $(r,\rho,\theta,z)\mapsto (r+a,\rho,\theta,z+b)$. Therefore, following the same strategy 
as in Definition \ref{bHolder}, we first decompose the space 
$$
X_{\text{loc}} := (c,1)\times [0,\rho_0)\times S^1\times \bbT^N \ni (r,\rho,\theta,z)
$$
into subdomains which are equivalent to one another via these particular dilations and translations. This is achieved by a Whitney cube decomposition. Indeed, for an appropriate 
sequence of centers $(r_i,z_i)$ and dyadic radii $2^{-j}$, we define the (overlapping) Whitney cubes 
$$
Q_{i,j} := \{(r,\rho,\theta,z): |r-r_i|\leq 2^{-j+1},\ \ 2^{-j-1}\leq \rho\leq 2^{-j+1},\ \ \theta\in S^1,\ \ |z-z_i|\leq 2^{-j+1}\}.
$$
The actual `center' of $Q_{i,j}$ is $(r_i,2^{-j},z_i)$. As in the $b$-Hölder setting, note that each $Q_{i,j}$ may be transformed to the standard cube of 
unit size, $Q_{0,1}$, via an appropriate map $D_{i,j}:Q_{0,1}\to Q_{i,j}$, which is a composition of a dilation and translation of the form described above. 
Thus, any function $u$ then corresponds to a sequence of functions $u_{i,j} := D_{i,j}^*(u|_{Q_{i,j}}) \equiv (u|_{Q_{i,j}})(D_{i,j}^{-1}(\cdot))$.

\begin{defn}\label{eHolder}
The space $\calC^{0,\alpha}_{e}((c,1)\times [0,1)\times S^1\times\bbT^N)$ consists of all functions for which 
$$
\|u\|_{e;0,\alpha}:= \sup_{i,j}\|u_{i,j}\|_{0,\alpha}<\infty.
$$
\end{defn}
The norm on the right-hand side here is the standard Hölder norm on $Q_{0,1}$. As before, we define $\calC_e^{k,\alpha}$ to consist of those functions 
such that up to $k$ derivatives with respect to combinations of $\rho\del_\rho$, $\rho\del_r$, $\rho\del_{z_j}$, and $\del_\theta$, lie in 
$\calC^{0,\alpha}_e$. Finally, we also define the weighted spaces $\rho^\nu \calC^{k,\alpha}_e$.

It is immediate from the definition that
$$
\Delta:\rho^\nu \calC^{2,\alpha}_e(X_{\text{loc}};\calI)\longrightarrow \rho^{\nu-2}\calC^{0,\alpha}_e(X_{\text{loc}};\calI)
$$
is bounded for any $\nu\in \R$.

We now come to the analogous definitions on the entire space $B^3\times \bbT^N$, where we also take into account the degeneracy at $r=0$. For 
this observe that 
$$
r^2\sin^2(\rho)\Delta = \sin^2(\rho)\left(r^2\del_r^2+2r\del_r\right) + \left[\sin^2(\rho)\del_\rho^2 + \sin(\rho)\cos(\rho)\del_\rho + 
\del_\theta^2\right] + r^2\sin(\rho)^2\Delta_z
$$
and note that this is an elliptic combination of the vector fields $\rho r\del_r$, $\rho\del_\rho$, $\del_\theta$, and $\rho r\del_{z_i}$, with coefficients 
which are smooth on $\{r\geq 0,\rho\geq 0\}$. It is also useful to write 
\begin{equation}\label{2nd}
	r^2\sin^2(\rho)\Delta = \sin^2(\rho)\left(r^2\del_r^2 + 2r\del_r\right) + \sin^2(\rho)\Delta_{S^2} + r^2\sin^2(\rho)\Delta_z,
\end{equation}
where $\Delta_{S^2}$ is the Laplacian on $S^2$ (acting on sections of $\calI$).

The upshot is that we can treat \eqref{2nd} as analogous to \eqref{1st}, but where the `regular' operator is replaced by the singular operator $\Delta_{S^2}$. 
In other words, we can treat \eqref{2nd} as an edge operator, where the former compact cross-section $S^1$ is now replaced by the singular cross-section
$(S^2,\calI)$ due to the singularity at $\rho=0$. We can therefore give an almost identical definition of the \emph{iterated edge Hölder spaces}. Indeed,
first decompose the full space $B^3\times\bbT^N$ into a union of Whitney cubes 
$$
\widetilde{Q}_{i,j}:= \{(r,y,z)\in [0,1)\times S^2\times \bbT^N: 2^{-j-1}\leq r\leq 2^{-j+1},\ \ y\in S^2,\ \ |z-z_i|\leq 2^{-j+1}\}
$$
(here, $z_i$ is as before). Notice that up to a map $D_{i,j}:\widetilde{Q}_{0,1}\to \widetilde{Q}_{i,j}$ made up from the composition of an appropriate 
dilation $(r,y,z)\mapsto (\lambda r,y,\lambda z)$ and translation $(r,y,z)\mapsto (r,y,z+b)$, these are all equivalent to the standard cube 
$\widetilde{Q}_{0,1}$. Again, write $u_{i,j}$ for the pullback of $u|_{\widetilde{Q}_{i,j}}$ from $\widetilde{Q}_{i,j}$ to $\widetilde{Q}_{0,1}$.
\begin{defn}\label{ieHolder}
The space $\calC^{0,\alpha}_{\ie}(B^3\times\bbT^N)$ consists of all functions for which
$$
\|u\|_{\ie;0,\alpha} := \sup_{i,j}\|u_{i,j}\|_{b;0,\alpha}<\infty.
$$
\end{defn} 
Note that the norms on the right-hand side are the $b$-Hölder norms on $\widetilde{Q}_{0,1}$. In other words, we are taking advantage of the 
iterative structure of the singularity to define these function spaces in terms of the simpler $b$-Hölder spaces at the previous `level' of degeneracy. 
We can of course also define $\calC^{k,\alpha}_{\ie}$ to consist of those functions for which up to $k$ derivatives with respect to combinations of
$\rho r\del_r$, $\rho \del_\rho$, $\del_\theta$, and $\rho r\del_{z_i}$ lie in $\calC^{0,\alpha}_{\ie}$, and finally also the weighted spaces 
$r^\mu\rho^\nu \calC^{k,\alpha}_{\ie}$.

As before, the map
\begin{equation}\label{mapiedge}
\Delta : r^\mu\rho^\nu \calC^{2,\alpha}_{\ie}(B^3\times \bbT^N;\calI) \longrightarrow r^{\mu-2}\rho^{\nu-2}\calC^{0,\alpha}_{\ie}(B^3\times\bbT^N;\calI)
\end{equation}
is bounded.

\begin{remark} In Definition \ref{eHolder} and Definition \ref{ieHolder} (and the corresponding use of Definition \ref{bHolder}) the Hölder semi-norm 
of the two-valued function involves the difference of values of $u$ at two nearby points. This is interpreted in the usual manner, namely by taking 
the minimum of the Hölder quotient between when the two points are on the same branch or on different branches.
\end{remark}

\smallskip

We conclude this subsection by stating the main Fredholm and regularity result we use for this rest of this section.  This is a generalization
of Proposition \ref{prop:fredholm}. However, the extension from the theory of $b$-pseudodifferential operators to the broader class of
iterated edge pseudodifferential operators is not straightforward.  A discussion of the parametrix construction for ``depth $2$'' iterated
edge operators, a class sufficiently broad to include the case studied here, may be found in \cite{MazWit}, but the forthcoming
article \cite{AMU} will contain all details in a form immediately recognizable as covering the present circumstances. 

\begin{prop}\label{polyprop}
Fix $\mu \in (\gamma_1^-,\gamma_1^+)$ and $\nu\in (-3/2, 3/2)$. Then the map
$$
\Delta: r^\mu\rho^\nu \calC^{2,\alpha}_{\ie,D, \Gamma}(B^3\times \bbT^N;\calI)\longrightarrow r^{\mu-2}\rho^{\nu-2}\calC^{0,\alpha}_{\ie,\Gamma}(B^3\times\bbT^N;\calI)
$$
is an isomorphism. Furthermore, if $f\in r^{\mu-2}\rho^{\nu-2}\calC^{0,\alpha}_{\ie,\Gamma}$ is polyhomogeneous at the branching 
set $\mathcal{R}\times \bbT^N$, then the solution $u$ to $\Delta u = f$ is also polyhomogeneous at $\mathcal{R}\times\bbT^N$. 
\end{prop}
In particular, if $f$ is compactly supported away from $\mathcal{R}\times\bbT^N$, then $u$ is polyhomogeneous at $\mathcal{R}\times\bbT^N$.

In this setting, polyhomogeneity must be interpreted as follows.  The space $B^3 \times \bbT^N$ has an iterated edge singularity, with 
deepest stratum $\{0\} \times \bbT^N$ and intermediate strata $R_j \times \bbT^N$.  This can be blown up to a manifold with corners,
first replacing $B^3 \times \bbT^N$ by $S^2 \times [0,1] \times \bbT^N$ (namely, blow-up the origin in each $B^3$), and then blow
up each of the submanifolds $R_j \times \bbT^N$.  There are now two types of boundary faces: the first, obtained in the first step,
is the set where the radial variable $r$ in $B^3$ vanishes. This is a copy of $S^2$, but blown up at each of the four points $p_j$.
The second step also introduces four new boundary faces, which are the spherical normal bundles of $R_j \times \bbT^N$.  As we 
have been doing above, the variable $\rho$ can be chosen locally near any one of these faces and that face is defined by $\rho = 0$.
In other words, this two-step blow-up of $B^3 \times \bbT^N$ is a manifold with corners of codimension $2$, with boundary
defining function $r$ and $\rho$, and polyhomogeneity then means the existence of polyhomogeneous expansions as either
$r \to 0$ or $\rho \to 0$, and product type polyhomogeneous expansions as both $r, \rho \to 0$. 

This regularity statement is local, namely if $f$ is polyhomogeneous in some relatively open subset $U\cap (\mathcal{R}\times\bbT^N)$, then 
$u$ is polyhomogeneous in that same set.

\subsection{Minimal submanifolds with stratified branching set}\label{sec:non-linear-perturbation}

We now combine the results of the last subsections to produce branched minimal hypersurfaces as two-valued graphs of solutions to the minimal
surface equation over $B^3$ which are small in a suitable norm. In fact, it only involves small notational changes to allow $u$ to take values in 
some $\R^d$, which then yields branched minimal submanifolds of codimension $d$ in $B^3\times\bbT^N\times\R^d$. 
In fact, the reader will see that the modifications needed to treat the case $d > 1$ are only notational, so for simplicity of
presentation, we state the result for arbitrary $d$, but present the proof assuming $d=1$.

The main theorem of this section is the following. Recall that $\Gamma\subset O(3)$ is the symmetry group of the spherical tetrahedron acting 
on $S^2$, with corresponding spherical triangles $\{\mathcal{T}_\ell\}_{\ell=1}^4$ with vertices $\{p_\ell\}_{\ell=1}^4$ and rays 
$\{\mathcal{R}_\ell\}_{\ell=1}$, with $\mathcal{R} = \cup_\ell \mathcal{R}_\ell$. 
\begin{theorem}\label{thm:main-perturbative}
There exists an $\eps  > 0$, depending on the geometric and analytic data, such that if $\phi\in \calC^{2,\alpha}_{e,\Gamma}(S^2\times\bbT^N;\calI)$
has $\|\phi\|_{e;2,\alpha}<\eps$, then there exists a unique section $u\in \calC^{2,\alpha}_{\ie,\Gamma}(B^3\times\bbT^N;\calI)$ 
such that
\[\calH(u) = 0 \qquad \text{and} \qquad u|_{r=1}=\phi.
\]
Here $\mathcal{H}$ is the minimal surface operator. This solution $u$ branches along $\mathcal{R}\times\bbT^N$, 
and vanishes to order at least $3/2$ at $(\mathcal{R}\setminus\{0\})\times\bbT^N$ and to order at least $3/2+\tau$ for some $\tau > 0$ 
at $\{0\}\times\bbT^N$.  
\end{theorem}
\begin{remark}
Note that we are using the edge spaces $\calC^{2,\alpha}_e$ for $\phi$ here rather than the $b$ spaces. This is only because of the extra torus
factor, so the branching set in $S^2 \times \bbT^N$ is a collection of submanifolds rather than a discrete set.  When the torus factor is absent,
the correct hypothesis is that $\phi \in \calC^{2,\alpha}_{b,\Gamma}$.
\end{remark}
\smallskip
The argument follows the same lines as in Theorem \ref{CHS}, but is even simpler since the problem is unobstructed, so there is no need to restrict 
to a proper subspace of boundary values (with the appropriate symmetries).
\begin{proof}
Using the eigendata for the Laplacian with Dirichlet conditions on $\calT_1 \times \bbT^N$, or equivalently for the Laplacian on $S^2 \times \bbT^N$
on  the subspace of functions satisfying the odd periodicity condition and decaying like $\rho^\nu$ at each $p_j$, we can define the Poisson operator
\[
\calP: \rho^\nu \calC^{2,\alpha}_{e,\Gamma}(S^2 \times \bbT^N; \calI) \longrightarrow r^\mu \rho^\nu \calC^{2,\alpha}_{\ie,\Gamma}(B^3 \times \bbT^N;\calI),
\]
where $v = \calP(\phi)$ is the unique solution to $\Delta v = 0$, $v|_{r=1} = \phi$ satisfying the same odd periodicity condition and 
with $|v| \leq C r^\mu$ as $\mu \to 0$.  Here $3/2<\mu<\gamma_1^+$ and $0<\nu<3/2$.  In fact, because of the non-linear nature of
the problem, it is important to restrict $\nu$ to lie in the interval $(1, 3/2)$.

We next invoke the regularity theory of edge operators \cite{Maz} along with the $\Gamma$-equivariance to assert that along the rays $\mathcal R_j$, i.e., 
as $\rho \to 0$, this solution $v$ admits a polyhomogeneous expansion of the form
\[
v \sim \sum_{j=1}^\infty \rho^{j+1/2} e^{i (j + 1/2) \theta} v_j(r,z).
\]  
The uniformity of this vanishing as $r \to 0$ requires the regularity theory in the iterated edge calculus.  This is a somewhat lengthier
argument, but is carried out in full in \cite{MazWit} in a setting which is even slightly more general than the one here. 

Next, observe that by homogeneity considerations, if $1 < \nu < 3/2$, and if $\delta>0$ is sufficiently small, the minimal surface operator 
defines a smooth map 
$$
\mathcal{M}:r^\mu\rho^\nu \calC^{2,\alpha}_{\ie,\Gamma}(B^3\times\bbT^N;\calI)\cap \{u:\|u\|_{\ie;2,\alpha,\mu,\nu}<\delta\}
\longrightarrow r^{\mu-2}\rho^{\nu-2}\calC^{0,\alpha}_{\ie,\Gamma}(B^3\times\bbT^N;\calI).
$$
Finally, the Jacobi operator is simply the Laplacian on $B^3\times\bbT^N$ (if $u$ was instead $\R^d$-valued, the Jacobi operator would instead be $d$ copies of the Laplacian) and 
by Proposition \ref{polyprop} we know that
$$
\Delta:r^\mu\rho^\nu \calC^{2,\alpha}_{\ie,D,\Gamma}(B^3\times\bbT^N;\calI) \longrightarrow r^{\mu-2}\rho^{\nu-2}\calC^{0,\alpha}_{\ie,\Gamma}(B^3\times\bbT^N;\calI)
$$
is bijective. Now consider the map
$$
\widehat{\mathcal{M}}:\calC^{2,\alpha}_\Gamma(S^2\times\bbT^N;\calI)\oplus r^\mu\rho^\nu \calC^{2,\alpha}_{\ie,D,\Gamma}(B^3\times\bbT^N;\calI) 
\longrightarrow r^{\mu-2}\rho^{\nu-2}\calC^{0,\alpha}_{\ie,\Gamma}(B^3\times\bbT^N;\calI)
$$
defined by
$$
\widehat{\mathcal{M}}(\phi,v):= \mathcal{M}(\mathcal{P}(\phi) + v)
$$	
The differential of $\widehat{\mathcal{H}}$ in the second factor at $(0,0)$ is $\Delta$ and is thus bijective. Hence the implicit 
function theorem proves the existence of a smooth map 
$$
\mathcal{G}: \calU \longrightarrow r^\mu\rho^\nu \calC^{2,\alpha}_{\ie,D,\Gamma}(B^3\times\bbT^N;\calI),
$$
where $\calU$ is a neighborhood of $\phi=0$ in $\calC^{2,\alpha}_\Gamma(B^3\times\bbT^N;\mathcal{I})$, such that
$$
\widehat{\mathcal{M}}(\phi,\mathcal{G}(\phi)) \equiv 0.
$$
(Notice that in this exact conic/product situation, the fact that the range of $\calP$ is complemented is straightforward and does not require
an extra argument.)
This produces the family of solutions $u:=\mathcal{P}(\phi) + \mathcal{G}(\phi)$ satisfying all the conclusions of the theorem. 

As usual with the implicit function theorem, all solutions which are graphs of sufficiently small two-valued functions are obtained this way.
\end{proof}

Solutions constructed in Theorem \ref{thm:main-perturbative} enjoy better regularity properties.
\begin{prop}
The solutions $u$ constructed in Theorem \ref{thm:main-perturbative} are polyhomogeneous at the branching locus $\mathcal{R}\times\bbT^N$ and at $0$. More precisely,
near $R_j \times \bbT^N$, $u$ has an expansion of the form 
\[
u \sim \sum_{m=0}^\infty A_m(r, z) \rho^{ i (3m + 3/2) \theta},
\]
where we have chosen the angular coordinate $\theta$ to vanish on one of the edges of a triangle $T$ having a vertex at $p_j$. Furthermore, both $u$
and each coefficient $a_\ell$ in this expansion as $\rho \to 0$ have polyhomogeneous expansions 
\[
u \sim \sum_\ell B_\ell(\rho, \theta, z) r^{\gamma_\ell^+}, \quad A_m(r,z) \sim \sum_\ell C_\ell(z) r^{\gamma_\ell^+}
\]
with $B_\ell$ polyhomogeneous as $\rho \to 0$ and $C_\ell$ smooth. 
Observe that the exponents $3m + 3/2$ in the expansion as $\rho \to 0$ are precisely those for which $\Gamma$-equivariance holds. 
\end{prop}
\begin{proof}
We first address polyhomogeneity at each branching locus $R_j \times \bbT^N$ away from $\{0\}\times\bbT^N$.   We sketch the proof, which is closely modelled
on one used already in many situations (see \cite{Maz-rsyp} for example).  For simplicity, we omit the $\bbT^N$ factor since this only complicates the notation
but does not cause any difficulties, and also restrict to the codimension $1$ case. 

Recall the precise form of the minimal surface operator
\[
\calH(u) = (1 + |\nabla u|^2) \Delta u - \sum_{i,j} u_i u_j u_{ij}
\]
Recalling that $|u| \leq C \rho^\nu$ for some $1 < \nu < 3/2$, we have that $|\nabla u| \leq C \rho^{\nu-1}$, and hence both $|\nabla u|^2 \Delta u$ 
and $u_i u_j u_{ij}$ blow up like $\rho^{2\nu-2 + \nu-2} = \rho^{3\nu - 4}$, which is smaller than $\rho^{\nu-2}$ as $\nu>1$. In other words, we may treat these like
perturbation terms. 

Now, choose a edge parametrix $G$ for $\Delta$. This is a particular sort of pseudodifferential operator which is adapted to the singular structure of
this problem. The paper \cite{Maz} is a comprehensive reference, but for now, we only need to know a few things about this operator.  First, it
satisfies $G \circ \Delta = I - Q$ where $Q$ is residual, i.e, maps any space $\rho^\sigma \calC^0$ to $\rho^N \calC^\infty$ for any $N \geq 0$, and
\[
G: \rho^{\nu-2} \calC^{0,\alpha}_{e,\Gamma} \longrightarrow \rho^\nu \calC^{2,\alpha}_{e,\Gamma}
\]
along $R_j$, at least away from $r=0$.   Rewrite the minimal surface equation $\mathcal{M}(u) = 0$ as $\Delta u = \sum u_i u_j u_{ij} - |\nabla u|^2 \Delta u$ and apply $G$ to this.
We deduce that $u = Qu + G( \sum u_i u_j u_{ij} - |\nabla u|^2 \Delta u)$.  The right hand side is obviously in $\rho^\nu \calC^{2,\alpha}_{e,\Gamma}$ again,
but in fact there is a slight improvement because this non-linear combination of derivatives of $u$ is bounded by $C \rho^{ \nu-2 + \delta}$ for some $\delta > 0$,
as observed above.   We write this simply as $u = Qu + G f$ where $f \in \rho^{\nu-2+\delta} \calC^{0,\alpha}_{e,\Gamma}$. 

The key idea of \cite{Maz-rsyp} is that this extra factor of $\rho^\delta$ in $f$ can be parlayed into extra regularity with respect to the tangential
variable $r$ (or $(r,z)$ if $\bbT^N$ is included) in the output $Gf$. This involves a further investigation into the mapping properties
of pseudodifferential edge operators, but ultimately reduces to a commutator argument.   The upshot is that after a bootstrap, one obtains full
tangential regularity, so in particular $|(\rho\del_\rho)^j \del_r^p \del_\theta^q u| \leq C \rho^\nu$ for all $j, p, q$, i.e.~$u$ is conormal (cf.~Definition~\ref{polyhom}). The final step is then to produce an expansion. For this we simply write the minimal surface
equation $\calH(u) =0 $ as an ODE in the radial variable $\rho$ with coefficients which may involve derivatives with respect to the tangential
variables. The conormality shows that these derivatives may be considered as `lower order', and we can then derive an expansion for
this ODE in the normal variables, with the tangential variables carried along as parameters.   This produces some expansion for $u$,
and we can then put this equation back into the equation $\calH(u) = 0$ to see which of the terms in this expansion are actually needed,
versus those which cannot actually occur in this expansion.

This argument leads to a full proof of polyhomogeneity along the $R_j \times \bbT^N$. It remains to show that there is an expansion as $r \to 0$,
and for this a similar but more involved inductive bootstrapping must be done.  As noted above in the linear setting, this type of two-step polyhomogeneity 
is proved in a related context in \cite{MazWit}.  This proves the polyhomogenity in both $\rho$ and $r$. 
\end{proof}

\section{Branched minimal hypersurfaces with large boundary data}\label{sec:hypersurface}
In this section we construct solutions of the minimal surface equation in codimension one which may have `large' boundary data, 
in contrast to the small-data solutions constructed in Section \ref{sec:perturbative} (which were also in arbitrary codimension). The argument here
is elementary, using only barriers and symmetry considerations.

We now work in the more general setting where $\Gamma$ is a finite group of rotations acting appropriately on the $(n-1)$-sphere $S^{n-1}$ 
and, as before, we will produce solutions with domain $B^n\times\bbT^N$, for $\bbT^N$ the standard $N$-torus (here, $N\geq 0$).

\subsection{Setup}

We consider an isometric tiling of $S^{n-1}$ by convex isometric polytopes with (orientation preserving) symmetry group $\Gamma\subset 
\text{SO}(n)$. Such tilings exist in all dimensions, for instance those associated to the Platonic solids in $\R^3$. The simplest example in 
general dimensions, and the one we shall primarily consider for the sake of concreteness, is the tiling formed by the spherical projection of 
the standard $n$-simplex. Even more concretely, a key example is the tetrahedral tiling of $S^2$ (which has already been discussed at 
length in Section \ref{sec:perturbative}).

Thus, fix $n\geq 3$ and let $P\subset S^{n-1}$ be a closed convex spherical polytope contained within an open hemisphere such that the 
spherical reflections across each of the codimension one faces of $P$ generate a discrete group of isometries $\Gamma\subset \text{SO}(n)$ 
and that the set of images $\{\gamma(P):\gamma\in\Gamma\}$ tiles the sphere by isometric polytopes with disjoint interiors. We denote 
this tiling of $S^{n-1}$ by $\mathcal{T}^{(n-1)}$, and list its distinct elements as $\{P_j\}_{j=1}^G$.

For $0\leq\ell\leq n-1$ write $\mathcal{T}^{(\ell)}$ be the $\ell$-skeleton of this tiling, i.e.~the collection of all $\ell$-dimensional edges 
of the $P_j$. Each $\mathcal{T}^{(\ell)}$ is a finite collection of $\ell$-dimensional convex spherical polyhedra, and every element of 
$\mathcal{T}^{(\ell-1)}$ is contained in at least two elements of $\mathcal{T}^{(\ell)}$. We also assume that the subset
$$\mathcal{T}^{(n-3)}_{\text{odd}} := \{e\in \mathcal{T}^{(n-3)}: \#\{j: e\in \del P_j\} \text{ is odd}\}$$
is non-empty. Indeed, if $e\in \mathcal{T}^{(n-3)}_{\text{odd}}$ and $P_1$ is a polytope containing $e$, then there is a cyclic sequence 
consisting of an odd number of reflected images of $P_1$ which constitute the entire set of $(n-1)$-faces containing $e$ (this can be 
seen by restricting to a copy of $S^2$ perpendicular to $e$ at any point of its interior).

\begin{remark} The regular simplicial tiling of $S^{n-1}$ from the standard simplex has $\mathcal{T}^{(n-3)}_{\text{odd}} = 
\mathcal{T}^{(n-3)}$, so is always non-empty. Indeed, denoting by $v_0,\dotsc,v_n$ the unit vectors in $\R^n$ which are the 
vertices of the regular $n$-simplex, the faces of this tiling are in bijective correspondence with the set of ordered $n$-tuples 
$\{i_1<i_2<\cdots<i_n\}$, where $i_j\in \{0,\dotsc,n\}$, or, dually, with the complements of these sets, namely with the vertices 
which are the `polars' of each face. By this dual enumeration, the $(n-3)$-skeleton consists of those faces which are dual to the 
set of triples $\{i_1,i_2,i_3\}$ of distinct indices, and the isotropy subgroup of $\Gamma$ which fixes a given element $e\in 
\mathcal{T}^{(n-3)}$ induces the full symmetric group action on the corresponding triple, and hence $\#\{j:e\in \del P_j\} = 3$ is odd.
\end{remark}

\medskip

We now pass to the tiling $\widehat{\mathcal{T}}^{(n)}$ of the unit ball $B^n$, where each element of this tiling is the (unit) cone $C_1(P)$ over some $P\in\mathcal{T}^{(n-1)}$. Thus, each $C_1(P)$ is a convex spherical sector. The $(\ell+1)$-skeleton $\widehat{\mathcal{T}}^{(\ell+1)}$ of this 
tiling of $B^n$ is simply 
$$
\widehat{\mathcal{T}}^{(\ell+1)} = \{C_1(V): V\in \mathcal{T}^{(\ell)}\}.
$$
Thus, we are particularly interested in the $n$-dimensional elements $C_1(P_j)$, their (flat) faces $C_1(F)$ for $F\in\mathcal{T}^{(n-2)}$,
and the `spines' $C_1(e)$ for $e\in \mathcal{T}^{(n-3)}$.
Observe that the boundary of each $C_1(P)$ is a collection of elements of $\widehat{\mathcal{T}}^{(n-1)}$ with the outer `cap' $P\subset S^{n-1}$.

Our aim is to construct two-valued minimal graphs over $B^n\times \bbT^N$. To do this, we first construct a single-valued solution $u_1$ solving the 
minimal surface equation over the product $C_1(P_1)\times\bbT^N$ which vanishes at all the codimension one faces $C_1(F)\times\bbT^N\subset C_1(P_1)\times\bbT^N$, 
$F\in \mathcal{T}^{(n-2)}$. Then we extend this solution by odd reflection across each such codimension 1 face $C_1(F)\times \bbT^N$. The function
obtained this way is only well-defined up to a sign. Indeed, given any sequence of reflections corresponding to a sequence of abutting 
sectors $C_1(P_j)\times\bbT^N$ terminating back at $C_1(P_1)\times\bbT^N$, the function obtained by this chain of successive reflection 
coincides with $\pm\gamma^*u_1$ depending on where the number of steps is even or odd (the sign is $+$ if the number of steps is even and $-$ 
if it is odd) for some $\gamma\in \Gamma_1$, where 
$\Gamma_1$ is the isotropy group of $C_1(P_1)$ in $\Gamma$. Thus, provided $|u_1|$ is $\Gamma_1$-invariant, this extended function gives a 
well-defined two-valued function on $B^n\times \bbT^N$, which branches along $\widehat{\mathcal{T}}^{(n-2)}_{\text{odd}}\times\bbT^N$ (where
$\widehat{\mathcal{T}}^{(n-2)}_{\text{odd}}\subset\widehat{\mathcal{T}}^{(n-2)}$ is the set of unit cones over $\mathcal{T}^{(n-3)}_{\text{odd}}$).

\smallskip

Our main theorem in this section is therefore:

\begin{theorem}\label{thm:main}
Fix an isometric tiling $\{P_\ell\}_\ell$ of $S^{n-1}$ with isometry group $\Gamma$ as above. Let $g\in \calC^2(S^{n-1}\times\bbT^N)$ be any two-valued function 
which obeys:
\begin{enumerate}
\item [\textnormal{(i)}] $g$ is odd under any reflection across the codimension one faces of the principle cells $C(P)$;
\item [\textnormal{(ii)}] $|g|$ is invariant with respect to the isotropy group $\Gamma_P$ of each $P\in\mathcal{T}^{(n-1)}\times\bbT^N$;
\item [\textnormal{(iii)}] $g = |\nabla g| = 0$ on $\mathcal{T}^{(n-3)}\times\bbT^N$.
\end{enumerate}
Then there exists a two-valued function $u\in \calC^{1,\alpha}(B^n\times \bbT^N)\cap \calC^\infty((B^n\setminus
\widehat{\mathcal{T}}^{(n-2)})\times\bbT^N)$ which obeys:
\begin{enumerate}
\item [\textnormal{(a)}] $\mathcal{H}(u) = 0$ strongly on the complement of $\widehat{\mathcal{T}}^{(n-2)}\times\bbT^N$ and weakly on 
$B^n\times\bbT^N$; 
\item [\textnormal{(b)}] $u|_{S^{n-1}\times\bbT^N} = g$;
\item [\textnormal{(c)}] $u$ is average-free (i.e.~$u$ is well-defined up to sign) and $|u|$ is invariant with respect to the isotropy group of each 
$C_1(P)$ for $P\in\mathcal{T}^{(n-1)}$.
\end{enumerate}
Furthermore, $u$ has branch locus $\widehat{\mathcal{T}}^{(n-2)}_{\textnormal{odd}}\times\bbT^N$. Here $\alpha \in (0,1)$ depends only on
 $n$, $N$, and the minimal dihedral angle of $P_1$ at its $(n-3)$-skeleton.
\end{theorem}
Again, $\mathcal{H}(u):= (1+|\nabla u|^2)\Delta u - u_{ij}u_i u_j$ is the (codimension one) minimal surface operator.

\medskip

Before continuing, we make some remarks.

\begin{enumerate}
\item [(1)] As an intermediate regularity statement, although we only find a solution in $\calC^{1,\alpha}$, it follows from \cite{SW16} that 
in fact $u\in \calC^{1,1/2}$. In fact, we will see a stronger polyhomogeneity statement in Proposition \ref{prop:poly-2}. It also follows that 
the branched minimal hypersurface associated to the graph of $u$ is stable for area under ambient deformations.
\item [(2)] When $n=2$, the $(n-3)$-skeleton of the spherical tiling is empty. However, one may still consider two-valued boundary data 
obtained by extending by odd reflection a given function $g$ on an arc of length $2\pi/k$ in $S^1$, where $k$ is odd. In this case, 
Theorem \ref{thm:main} gives an extension to a two-valued solution $u$ of the minimal surface equation on $B^2_1(0)\times\bbT^N$ 
which has a branch point at $\{0\}\times\bbT^N$. This recovers some of the examples constructed in \cite{SW07} (namely, those which 
are average-free).
\item [(3)] If $n\geq 3$ and $\mathcal{T}^{(n-3)}_{\text{odd}}=\emptyset$, then $u$ decomposes as two single-valued functions, each 
satisfying the minimal surface equation, both of which vanish with horizontal tangent at $0$. In this case, $0$ is not a branch point.
\item [(4)] When the branch set of $u$ is non-empty, one can consider the frequency value at any branch point. Indeed, $u$ must have 
a monotone frequency function at all branch points by \cite{SW16} (however, in our setting the symmetries of $u$ reduce this to a 
frequency function in the classical, single-valued, setting). Similarly to Section \ref{sec:perturbative}, the value of the frequency 
function at $0$ is determined by the Dirichlet eigenvalues of the Laplacian on the spherical polygon $P_1\subset S^{n-1}$. 
In the simplest example, for the tetrahedral tiling of $S^2$ and $N=0$, the constructed minimal surfaces can have a frequency value at $0$ which is not equal to 
a two-dimensional frequency (or an integer), and with branch sets consisting of the product of $4$ rays emanating from the origin 
which meet $S^2$ at the vertices of a regular tetrahedron.
\end{enumerate}

\subsection{Existence of minimal surfaces in convex polyhedra}

Our proof of Theorem \ref{thm:main} is based on giving an existence result for the minimal surface equation (with possibly large boundary data) in convex curvilinear polyhedral domains. This can be viewed as a (mild) extension of a classical result of Jenkins--Serrin \cite{JS} to domains which only need to have piecewise smooth boundaries (see also \cite[Theorem 12.9]{Giusti}). For simplicity of notation we do not include the extra torus factor here, but a quick inspection of the proof shows that the argument below can be modified to include that product case. 

The domains we consider here are the following.

\begin{defn}
We say an open bounded domain $\Omega\subset\R^n$ is a \emph{convex curvilinear polyhedron} if it can be written as a finite intersection
$$
\Omega = \bigcap^G_{j=1}\{\phi_j<0\}
$$
where each $\phi_j\in \calC^\infty(\R^n)$ is convex and furthermore, for any $p\in \del\Omega$, if $\phi_{i_1},\dotsc,\phi_{i_k}$ are all the $\phi_j$ 
which vanish at $p$, we assume:
\begin{enumerate}
\item [(a)] $\nu_j(p):= \nabla\phi_j(p)/|\nabla\phi_j(p)|$, for $j\in \{i_1,\dotsc,i_k\}$ are linearly independent;
and
\item [(b)] for any distinct $j,j^\prime\in \{i_1,\dotsc,i_k\}$, we have $\nu_j(p)\cdot\nu_{j^\prime}(p)\leq 1-\delta$ for some $\delta>0$ which is 
independent of $p$.
\end{enumerate}
\end{defn}
We define the regular part of the boundary, $\del_{\text{reg}}\Omega$, to consist of those points where exactly one $\phi_j$ vanishes. In general, 
we write $\del_\ell\Omega$ for the set of points where precisely $n-\ell$ of the $\phi_j$ vanish. In other words, $\del_\ell\Omega$ should be 
regarded as the $\ell$-skeleton of the boundary.

\smallskip

\textbf{Note:} When the $\phi_j$ are all affine, then $\Omega$ is a standard convex polyhedron.

\smallskip

Our main existence result on such domains is the following:

\begin{theorem}\label{thm:existence}
Let $\Omega\subset\R^n$ be a convex curvilinear polyhedron. Suppose that $g$ is a $\calC^2$ function defined on a neighborhood of $\del\Omega$ 
such that $g = |\nabla g| = 0$ on $\del_{n-2}\Omega$. Then there exists a unique solution $u\in \calC^\infty(\Omega)\cap 
\calC^{1,\alpha}(\overline{\Omega})$ such that:
\begin{enumerate}
\item [\textnormal{(i)}] $\mathcal{H}(u) = 0$ in $\Omega$ and $u=g$ on $\del\Omega$;
\item[\textnormal{(ii)}] $u = |\nabla u| = 0$ on $\del_{n-2}\Omega$, with $|\nabla u(z)|\leq C\dist(z,\del_{n-2}\Omega)^\alpha$.
\end{enumerate}
Here $\alpha = \alpha(\Omega)\in (0,1/2)$ and $C = C(\Omega,\|g\|_{\calC^2})$.
\end{theorem}

\begin{proof}
Uniqueness of the solution (even amongst the class of Lipschitz weak solutions) follows from the connectedness of $\Omega$ 
and the strong maximum principle since the difference $w$ of two solutions satisfies the equation $\div(A\nabla w) = 0$ for some uniformly 
elliptic matrix $A$ with bounded coefficients. 
	
Turning to the existence and the estimates, there are two ways to approach this; both are classical, but adapted to accommodate
that $\del\Omega$ has corners. 
	
The first is to construct a Lipschitz solution $u$ following \cite[Theorem 12.9]{Giusti}, and show that it is $\calC^{1,\alpha}$. For this it suffices to 
construct suitable Lipschitz supersolutions to $\mathcal{H}$ at any $p\in\del\Omega$. Namely, for each such $p$, we construct a function $w$ 
(which may depend on $p$, but whose Lipschitz constant does not) which satisfies:
\[
w\geq g\quad \text{on }\del\Omega,\qquad \mathcal{H}(w)\leq 0\quad \text{in }\Omega,\qquad \text{and}\qquad w(p) = g(p).
\]
Indeed, if we position the domain so that $\Omega\subset\{x_n>0\}$ and $p=0$, then we take $w(x) = g(x) +\psi(x_n)$, where
$\psi(x_n) := A^{-1}\log(1+Bx_n) - C(n)A x_n$, $x_n < r$, for some $A> \|g\|_{\calC^2}$ and suitable constant $B$.   This family of
supersolutions leads, in the usual way, to a Lipschitz solution $u$ on $\Omega$.  Standard elliptic regularity results give that $u$ lies in
$\calC^\infty(\Omega)$, but is $\calC^{1,\beta}$ for all $\beta < 1$ up to $\del_{\text{reg}}\Omega$ and Lipschitz near the corners.
	
In the second approach, we choose a family $\Omega_\eps$ of smooth, strictly convex, domains which lie inside $\Omega$ and such that
$\cup_{\eps>0}\Omega_\eps=\Omega$. We assume that $\del\Omega_\eps$ converges smoothly to $\del\Omega$ away from $\del_{n-2}\Omega$. 
By the classical result of Jenkins--Serrin \cite{JS}, there exists a function $u_\eps\in \calC^\infty(\Omega_\eps)\cap \calC^2(\overline{\Omega}_\eps)$ 
such that
$$
\begin{cases}
\mathcal{H}(u) = 0 & \text{on }\Omega_\eps\\
u_\eps = g & \text{on }\del\Omega_\eps.
\end{cases}
$$
It then follows by barrier techniques that $u_\eps$ converges to a limiting solution $u$ to $\mathcal{H}(u) = 0$ in $\Omega$ which is Lipschitz 
on $\overline{\Omega}$ and $\calC^{1,\beta}$ for all $\beta < 1$ up to $\del_{\text{reg}}\Omega$. 
	
With either argument, it remains to show that $u=|\nabla u|=0$ on $\del_{n-2}\Omega$ 
and $|\nabla u(z)|\leq C\dist(z,\del_{n-2}\Omega)^\alpha$. For this, we may argue locally near each point $p\in \del_{n-2}\Omega$, and we do so by 
constructing a local barrier.
	
As before, fix $p \in \del_{n-2}\Omega$ and position $\Omega$ to lie inside the sector 
$$
S_\alpha:=\{(x_1,x_2,x^\prime)\in \R^n: x_1=r\cos(\theta),\ x_2=r\sin(\theta),\ \text{for } r\geq 0\text{ and }|\theta|\leq\beta/2\}
$$
for some $\beta\in (0,\pi)$. Fix constants $\gamma\in (1,\pi/\beta)$ and $0 < \eps \ll 1$, and define
$$
v(r,\theta,x^\prime):= Br^\gamma\psi(\theta) \qquad \text{where}\qquad \psi(\theta):= \cos\left(\frac{\pi}{\beta+\eps}\theta\right).
$$
Then $v\geq Br^\gamma\psi(\beta/2)>0$ on $S_\beta.$   This function $v$ is simply the standard $2$-dimensional barrier for 
the slightly larger sector $S_{\beta+\eps}$, extended to be constant in the remaining directions. We then calculate that
$$
\H(v) = Br^{\gamma-2}\left(\gamma^2-\left({\pi}/{\beta}\right)^2\right)\psi(\theta) + O(r^{3\gamma-4}),
$$
and since $\gamma-2<3\gamma-4$ when $\gamma>1$, the first term dominates for small $r$, and has a negative coefficient. Thus, 
$\H(v)\leq 0$ in $S_\beta\cap \{r\leq r_0\}$.
	
Choosing $B \gg 0$ and $\eta>0$, we then have $v+\eta\geq u$ on $\del(\Omega\cap S_\beta\cap \{r\leq r_0\})$, and so by the maximum principle,
$v+\eta\geq u$ on $\Omega\cap S_\beta\cap\{r\leq r_0\}$. Similarly, $-v-\eta<u$ in this same region. Letting $\eta\searrow 0$ we deduce that
$$
|u(x)| \leq v \leq Cr^\gamma \quad \text{on }\Omega\cap S_\beta\cap \{r\leq r_0\},
$$
and in particular, $u = 0$ on $\del_{n-2}\Omega$. 

To prove the corresponding gradient estimate near $\del_{n-2}\Omega$, write $d(z):= \dist(z,\del_{n-2}\Omega)$.  Fix $d_0 > 0$ and consider any
$z_1\in \Omega$. Setting $d_1:= d(z_1)$ then we assume that $d_1 \leq d_0$ since the gradient estimate when $d_1\geq d_0$ is standard. 
With a constant $c>0$ to be chosen later, we consider the two regions $\dist(z_1,\del\Omega)\geq cd_1$ and $\dist(z_1,\del\Omega)\leq cd_1$
separately.
	
In the first region, rescale by setting $x := (z-z_1)/d_1$ and $\hat{u}(x):= d_1^{-1}u(z_1+d_1x)$. Then $\nabla_x \hat{u}(0) = \nabla_z u(z_1)$. 
Writing $\alpha:= \gamma-1$, since $|u(z)|\leq Cd_1^{1+\alpha}$ in $B_{cd_1}(z_1)$, we know that $\hat{u}(x)\leq Cd_1^\alpha$ in $B_c(0)$. Hence, 
by standard elliptic estimates,
$$
|\nabla_zu(z_1)| = |\nabla_x\hat{u}(0)| \leq \calC^\prime d_1^\alpha.
$$
For the second region, let $y$ be a point in $\del\Omega$ which attains $\dist(z_1,\del\Omega)$. Necessarily $z \in B_{2cd_1}(y)$, and this ball
intersects $\del\Omega$ only in its regular portion (provided $c<1/3$).   The function $\hat{u}(x):= (cd_1)^{-1}u(y+cd_1 x)$ solves the minimal 
surface equation in $B_2(0)$ and assumes the boundary values $\hat{g}(x):= (cd_1)^{-1}g(y+cd_1 x)$. Since $g\in \calC^2$ and $g=|\nabla g|=0$ 
on $\del_{n-2}\Omega$, we have that $|g|\leq Cd_1^2$ and $|\nabla_z g|\leq Cd_1$ on $B_{2cd_1}(y)$, hence $|\nabla_x\hat{g}|\leq Cd_1$. The usual 
(boundary) elliptic estimates for the minimal surface equation on this rescaled region then give $|\nabla_z u(z_1)|\leq \calC^\prime d_1^\alpha$, 
as before.

This completes the proof. 
\end{proof}

\subsection{Proof of Theorem \ref{thm:main}}
We now complete the proof of Theorem \ref{thm:main}. Again, for simplicity we do this only when the torus factor is absent. The extension of
the arguments to include that factor are straightforward and left to the reader.

\begin{proof}
Extend the function $g$ in the statement of the theorem to a neighborhood of the boundary by setting $h = \eta(r)g$ for some cutoff
function $\eta\in \calC^\infty_c((0,\infty))$ with $\eta\equiv 1$ on a neighborhood of $1$. By Theorem \ref{thm:existence}, applied with data $h$, 
there exists a (single-valued) function $u\in \calC^\infty(C_1(P_1))\cap \calC^{1,\alpha}(\overline{C_1(P_1)})$ such that
$$
\begin{cases}
\mathcal{H}(u) = 0 & \text{on }C_1(P_1),\\
u = g & \text{on }\del C_1(P_1).
\end{cases}
$$
Notice that by uniqueness of solutions, $u$ inherits all the same symmetries of $g$. Now one obtains a two-valued solution $\tilde{u}$ by repeated reflection 
across the codimension one faces $C_1(F)$.  Since $u$ vanishes along these faces, the smoothness of $\tilde{u}$ there is clear, and hence $\tilde{u}\in
\calC^\infty(B^n\setminus\widehat{\mathcal{T}}^{(n-2)})$. However, $u$ vanishes to some order $\gamma>1$ along each part of $\widehat{\mathcal{T}}^{(n-2)}$,
so in fact $\tilde{u}\in \calC^{1,\alpha}(B^n)$. Since $\tilde{u}$ solves the minimal surface equation strongly away from $\widehat{\mathcal{T}}^{(n-2)}$, 
and this $(n-2)$-skeleton has $2$-capacity zero, the graph of $\tilde{u}$ is necessarily stationary for the area functional. Moreover, we have determined
enough about the local structure near $\widehat{\mathcal{T}}^{(n-2)}_{\text{odd}}$ so that $\tilde{u}$ necessarily has branch points along this locus.
\end{proof}

One can check that the decay rates at $0$ of the solutions $u$ constructed in Theorem \ref{thm:main} are no smaller than the smallest eigenvalue of the 
Dirichlet problem for the Laplacian on the spherical polytope $P_1$. In fact, just as for the perturbative solutions obtained in Section \ref{sec:perturbative}, 
these solutions also enjoy greater regularity along their branch sets.

\begin{prop}\label{prop:poly-2}
Suppose $u$ is as in Theorem \ref{thm:main}. Then $u$ is polyhomogeneous along the entire branching locus $\widehat{\mathcal{T}}^{(n-2)}$, as well as 
at $0$.
\end{prop}

\begin{proof}
The proof is almost identical to the one sketched at the end of Section \ref{sec:non-linear-perturbation}. When $n=3$, the proof is identical, and uses
the fact that $|\nabla u| \leq C \rho^{\nu-1}$ for some $\nu \in (1, 3/2)$.   To prove this polyhomogeneity in the general higher dimensional
setting, one must work with iterated edge operators adapted to the higher depth stratification of the skeleton of the tiling of $S^{n-1}$. The overall
structure of the argument is the same, but this now relies on the full iterated edge pseudodifferential calculus developed in \cite{AMU}. 

We recall an observation made earlier that these expansions may include positive integral powers of $\log \rho$ as factors, where $\rho$ is
distance to some stratum. However, these never appear in leading coefficients, and hence do not play any role in any of the arguments.
\end{proof}

We note that the polyhomogeneous expansion along $(\widehat{\mathcal{T}}^{(n-2)}\setminus \widehat{\mathcal{T}}^{(n-3)})\times\bbT^N$ is of the usual 
form \eqref{asymp}, starting with terms $\rho^{3/2}$. The expansion at $0$ involves the terms $r^{\gamma_\ell^+}$, where the $\gamma_\ell^+$ are determined 
by the Dirichlet eigenvalues of $\Delta_{S^{n-1}}$ on $P$.

\section{Compact Examples}\label{sec:compact}

In this section, we demonstrate how ideas similar to those in the previous sections can be used to produce minimal hypersurfaces which are 
\emph{compact without boundary} and have stratified branch sets. We construct these as submanifolds of a warped product 
$S^n\times (-1,1)$, with coordinates $(y,t)$ and with metric 
\begin{equation}
g = dt^2 + f(t)^2\, g_{S^n}.
\label{warped}
\end{equation}
As before, this is done by taking advantage of the action of a finite group $\Gamma$ on $S^n$ and the corresponding tiling $\mathcal{T}^{(n)}
= \{P_j\}_{j=1}^G$ of $S^{n}$ by fundamental domains, where each $P_j$ is a convex polytope lying strictly in a hemisphere, exactly 
as in Section \ref{sec:hypersurface}.

In the previous sections we considered graphs of functions $B^n\to \RR$ as minimal submanifolds in $B^n \times \RR$ with the product metric, 
and the key to the existence theory was the flexibility allowed by imposing arbitrary boundary values at $\del B^n$.   We now construct minimal
submanifolds in $S^n \times (-1,1)$ as two-valued graphs over $S^n$, i.e., which correspond to $\ZZ_2$-functions $u: S^n \to (-1,1)$. 
Because of the maximum principle, it is not possible to find such multi-valued graphs which are minimal with respect to the product metric on 
$S^n\times (-1,1)$. The main observation in this section is that this becomes possible for a suitable warped product metric \eqref{warped}.

More explicitly, consider any $\Gamma$-equivariant function $u$ on $S^n$ with sufficiently small norm (to be specified below)
and its two-sheeted graph 
\[
Y_u = \{ \exp_y( u(y) \del_t): y \in S^n\}.
\]
Of course, this is equivalent to the graph of one sheet of $u$ over $P_1$, a convex spherical polytope in $S^n$ which
is a fundamental domain for the $\Gamma$ action, as long as $u$ vanishes on $\del P_1$, since such a one-sheeted  
graph can be extended by odd reflection.  Both points of view are useful below. 

Our goal is to show that there is a wide class of warping functions $f(t)$ for which this two-sheeted minimal graph exists. We note
that this does not necessarily violate the maximum principle because the warping function adds extra terms to the minimal
surface operator.

We do not treat the most general form of his problem, but instead consider the following restricted form.   Consider a smooth 
strictly positive {\it even} function $f(t)$ for $t\in (-1,1)$, with $f'(0) = 0$ and $f''(0) < 0$. Then $S^n \times \{0\}$ is 
a minimal hypersurface of $S^n \times (-1,1)$ with warped product metric \eqref{warped}.  Define $f_\lambda(t) = f(\lambda t)$, 
$\lambda \geq 1$, and denote the corresponding warped product metrics on $S^n \times (-1,1)$ by $g_\lambda$. 
Our claim is that as $\lambda$ gets larger, so that $(f_\lambda)^{\prime\prime}(0)$ becomes more negative, the Jacobi operator on 
$\Gamma$-equivariant functions becomes unstable. This leads to non-trivial solutions of the non-linear problem by
standard bifurcation theorems. 
 
Let us begin by calculating that if $u\in \calC^2_\Gamma(S^n)$ has small norm, then its normal graph $Y_u$ is a branched 
hypersurface in $S^n\times (-1,1)$, with area with respect to the warped product metric equal to
$$
\calA(u) := \int_{S^n}f(u)^{n-1}\sqrt{f(u)^2+|\nabla u|^2}\, dV;
$$
here $\nabla$ and $dV$ are the gradient and volume form for $S^n$, and.the integral should be taken for both sheets $\pm u$.  The 
first variation of this at $u=0$ gives the minimal surface equation
$$
\calH(u) = \div\left(\frac{\nabla u}{f(u)\sqrt{f(u)^2+|\nabla u|^2}}\right) - \frac{f^\prime(u)}{\sqrt{f(u)^2+|\nabla u|^2}}
\left(n-\frac{|\nabla u|^2}{f(u)^2}\right),
$$
and a short calculation, using that $f'(0) = 0$, gives that the Jacobi operator is
$$
Lw = D\mathcal{H}|_{u=0}(w) = \Delta w - n f^{\prime\prime}(0)w.
$$
In particular, for the family of dilated warping functions $f_\lambda$ introduced above, 
\[
L_\lambda = \Delta  - n\lambda^2 f^{\prime \prime}(0).
\]
These formul\ae\  can also be obtained from the standard formula $L = \Delta + \text{Ric}(\nu,\nu)+|A|^2$ for the Jacobi formula in any 
Riemannian manifold, where $A$ is the second fundamental form of the hypersurface. Here $S^n\times \{0\}$ is totally geodesic, so $A\equiv 0$,
and the sectional curvature of $g$ for any $2$-plane containing the unit normal $\del_t$ equals $-f^{\prime\prime}(0)$.

To state our main theorem, fix the warping function $f$ as above, and write the minimal surface and Jacobi operators associated to $g_\lambda$ 
as $\mathcal{H}_\lambda$ and $L_\lambda$. Now restrict $L_\lambda$ to act on $L^2_\Gamma(S^n)$, the subspace of $\Gamma$-equivariant functions.  
We use the Friedrichs extension of this operator,  starting from the core domain of smooth compactly supported functions with supports disjoint 
from the $(n-2)$-skeleton of of the union of the $\del P_j$ (corresponding to the vertices when $n=2$). 
By virtue of the $\Gamma$-equivariance, functions which vanish at least like $\rho^{3/2}$ at all points of this
$(n-2)$-skeleton are allowed. This Friedrichs extension of $-L_\lambda$ is self-adjoint, and has discrete spectrum with eigenvalues
\[
\mu_j(\lambda) = \mu_j(0) + n\lambda^2 f''(0), \ \ j \geq 0
\]
where $\{\mu_j(0)\}_j$ is the set of eigenvalues for $-\Delta$ acting on $L^2_\Gamma$.   Each $\mu_j(0) > 0$, hence $\mu_j(\lambda) > 0$ for 
$\lambda$ sufficiently small, depending on $j$. However, there is a downward spectral flow: as $\lambda$ increases, each $\mu_j(\lambda)$ 
eventually crosses $0$ and becomes negative.  For $j$ fixed, this occurs when $\lambda_j = \sqrt{-\frac{\mu_j(0)}{nf^{\prime\prime}(0)}}$. 
In any case, there is a sequence of values $\lambda_1< \lambda_2<\cdots\to +\infty$ such that $0$ is an eigfenvalue of $L_\lambda$ on 
$\Gamma$-equivariant functions if and only if $\lambda  \in \{\lambda_j\}_j$. 

When $0<\lambda<\lambda_1$,  there are no non-trivial solutions to $L_\lambda u = 0$ with $u = 0$ on $\del P_1$, and one does not 
expect any `small' minimal graphs over $P_1$ vanishing at $\del P_1$. However, bifurcations in the non-linear family occur 
whenever $\lambda$ passes one of the values $\lambda_j$. 

As we have done earlier, we consider these operators acting on weighted H\"older spaces rather than on $L^2$.    First consider the case $n=2$ 
(which thus produces branched minimal surfaces in $(S^2\times (-1,1), g_\lambda)$, objects which are similar enough to known examples
that their existence is no surprise at all). In this case
\begin{equation}
L_\lambda: \rho^\nu \calC^{2,\alpha}_{b,\Gamma}(S^2) \longrightarrow \rho^{\nu-2} \calC^{0,\alpha}_{b,\Gamma}(S^2)
\label{Gamma-isom}
\end{equation}
is Fredholm of index $0$ provided $\nu \in (-3/2, 3/2)$.  This means in particular that in order to check that this
map is an isomorphism, it suffices to show it is injective. 
If in addition $\nu > 1$, then there exists a neighborhood $\calU$ in this 
domain space such that 
\begin{equation}
\calH_\lambda: \calU \longrightarrow \rho^{\nu-2}\calC^{0,\alpha}_{b,\Gamma}(S^2)
\label{Gamma-nonlin}
\end{equation}
is a smooth mapping.   The key issue when applying the bifurcation theorems below is to identify the singular parameter values.
\begin{lemma}
When $n=2$, the set of values of $\lambda$ for which the differential in $u$ of 
\begin{align*}
&\RR^+ \times \rho^\nu \calC^{2,\alpha}_{b,\Gamma}(S^2)  \longrightarrow \rho^{\nu-2}
\calC^{0,\alpha}_{b,\Gamma}(S^2)\\
& \hspace{4.7em}(\lambda, u)\longmapsto \H_\lambda(u)
\end{align*}
is not an isomorphism is the same as the set of values $\lambda_j$ where $\{0\} \in \mathrm{spec}( -L_\lambda, L^2_{\Gamma})$.
\label{2dsingvalues}
\end{lemma}
The proof is immediate since the solutions of $-L_\lambda u = 0$ in $L^2_\Gamma$ necessarily lie in $\rho^\nu \calC^{2,\alpha}_{b,\Gamma}$. 

\begin{theorem}\label{thm:bifurcation}
For each $j\geq 1$, there is an $\eps_j>0$ such that for $\lambda\in (\lambda_j,\lambda_j+\eps_j)$ there exist non-trivial 
$\Gamma$-equivariant solutions $u_{j,\lambda}$ to the minimal surface equation on $S^2$ in the warped product $(S^2 \times (-1,1), g_\lambda)$;
these solutions bifurcate off the main branch $u\equiv 0$.
\end{theorem}

\begin{figure}[h]
    \centering
    \subfigure{
        \includegraphics[width=0.45\textwidth]{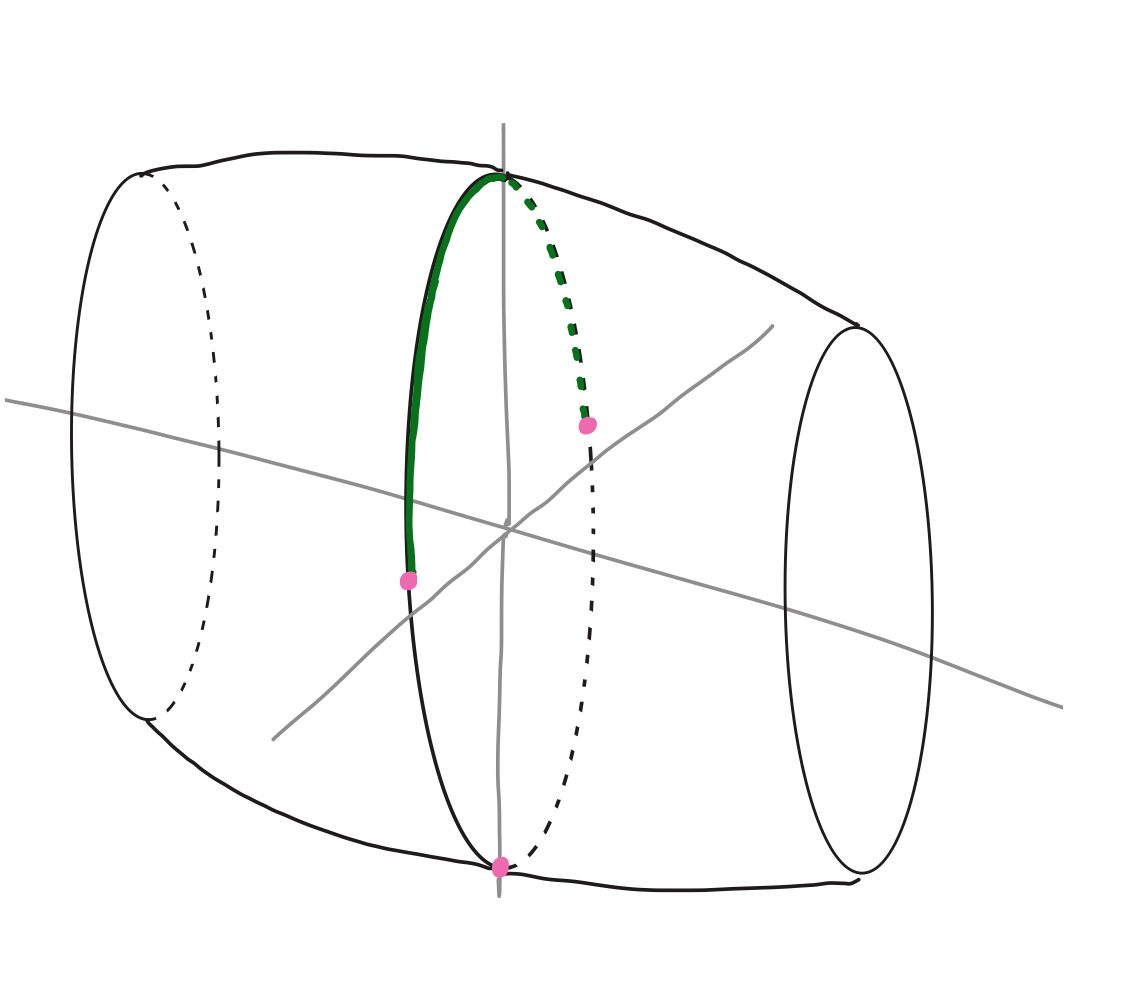}
    }
    \hfill
    \subfigure{
        \includegraphics[width=0.45\textwidth]{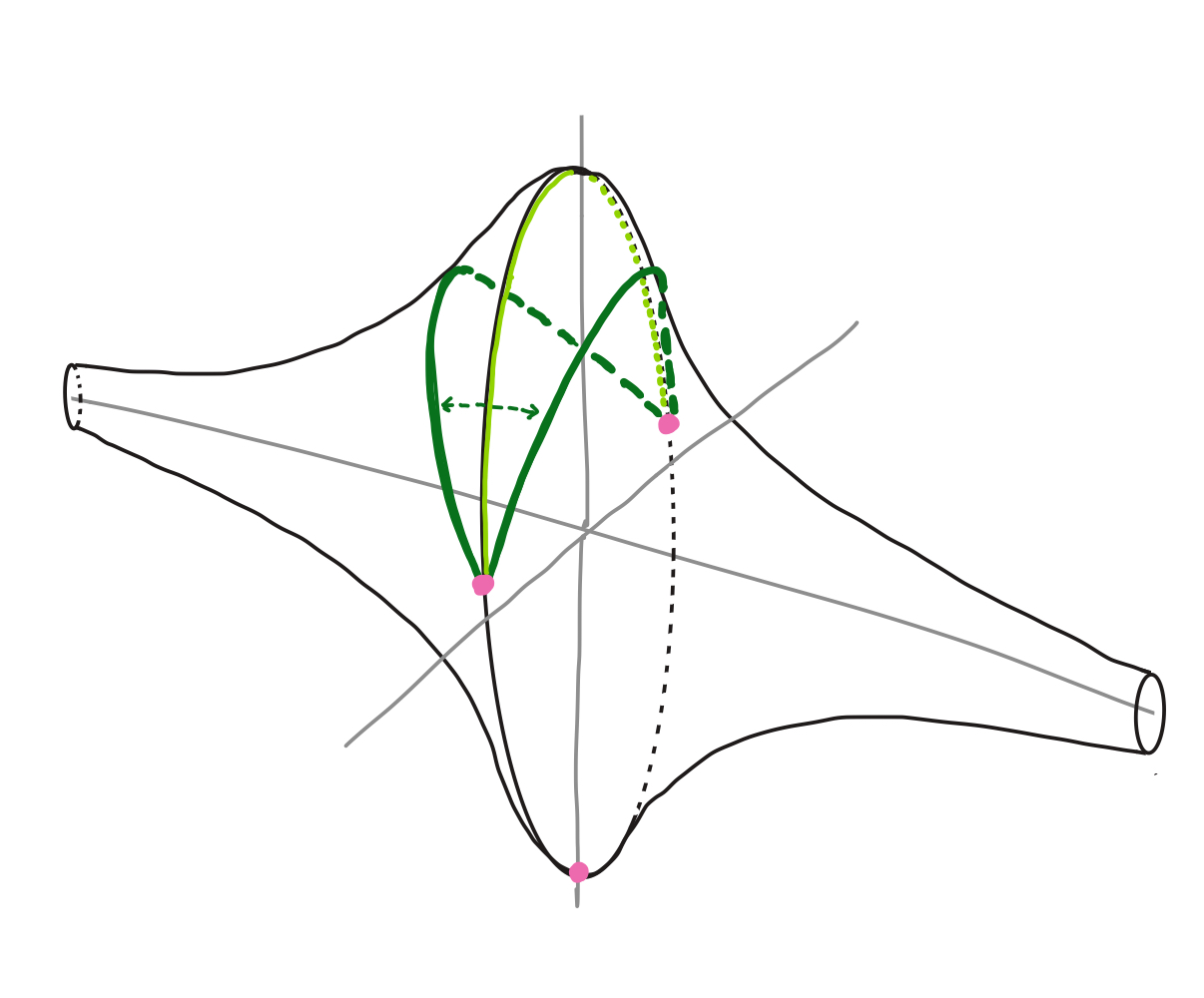}
    }
    \caption{\footnotesize The bifurcation argument for $S^1\times(-1,1)$ and $\Gamma=\Z_3$. On the left $\lambda$ is small and there is only one geodesic connecting the pink points, namely an equator. On the right $\lambda$ is large, so the equator becomes unstable and two stable geodesics appear (in dark green).}
    \label{fig:uv}
\end{figure}

\begin{proof}
The theorem is a consequence of two standard bifurcation theorems. The first (and stronger) of these is due to Crandall--Rabinowitz \cite{CR}. 
It states that, given the family of non-linear $\mathcal{H}_\lambda$, a bifurcation occurs whenever any one of the $\mu_j(\lambda)$ is a simple 
eigenvalue which crosses $0$ transversely as a function of $\lambda$. This occurs, in particular, for $\lambda=\lambda_1$, as indeed the 
smallest eigenvalue has multiplicity one and, if $\phi_1$ is the corresponding eigenfunction, the standard variation formula gives
$$
\left.\frac{d}{d\lambda}\right|_{\lambda=\lambda_1}\mu_1(\lambda) = \int_{S^2}\phi_1(\del_\lambda L_\lambda|_{\lambda=\lambda_1})\phi_1\, dV 
= -2n\lambda_1 f^{\prime\prime}(0)\int_{S^2}|\phi_1|^2\, dV \neq 0.
$$
(This computation can equally well be carried out for a single branch of $\phi_1$ on the fundamental domain $P_1$.) 
The result in \cite{CR} says even more: it asserts also that the new family of solutions $u_{1,\lambda}$ is smooth as a function of $\lambda 
\in [\lambda_1, \lambda_1 + \eps_1)$. 
	
A weaker, but more general, theorem due to Smoller--Wasserman \cite{SW} asserts that given any eigenvalue crossing,
without any assumption about simplicity or transversality of the crossing, there must exist new solutions $u_{j,\lambda}$, when 
$\lambda \in (\lambda_j, \lambda_j + \eps_j)$ for some sufficiently small $\eps_j$. Unlike the previous bifurcation theorem, 
one does not know that these lie in a smooth family of solutions.  This concludes the proof.
\end{proof}

We can reach the same conclusion in higher dimensions as well. The $n=2$ case has been presented first 
because the higher dimensional case involves a more intricate setup, mainly in terms of the function spaces and mapping
properties.  We describe this now.

Given the action of the discrete group $\Gamma$ on $S^n$, we can associate the tiling $\calT$ of the sphere by spherical 
convex polytopes, which we view as a simplicial decomposition.  As such, it has a skeleton $\calT^{(m)}$ of all
dimensions $0 \leq m \leq n$; the $n$-skeleton is simply the entire sphere, and the $(n-1)$-skeleton is the union
of boundary faces of all of the polytopes $P_j$.  Harmonic functions obtained by odd reflection are smooth across these
faces.  The new feature is that we must specify growth or decay at each of the deeper skeleta $\calT^{(m)}$, $m \leq n-2$.
As we have already seen when $m = n-2$, this involves studying the model operator at a point $p \in \calT^{(n-2)}$,
which is simply the Euclidean Laplacian on the sector of opening angle $2\pi/3$ in  the two-dimenional normal bundle to that stratum at $p$.  This model
determines a set of indicial roots $\{\pm\frac{1}{2}, \pm\frac{3}{2}, \ldots \}$, which are the possible rates of
blow up or decay of $\ZZ_2$-harmonic functions near $p$.  Imposing $\Gamma$-equivariance further winnows this
list to $\{\pm\frac{3}{2},\pm\frac{9}{2}, \ldots \}$, and the two key facts we used above in the $2$-dimensional case are
that the minimal surface operator is a smooth map from functions with small norm which decay at fixed rate $\nu \in (1, 3/2)$ to
functions blowing up like $\rho^{\nu-2}$, and furthermore, that the Jacobi operator $L_\lambda$ is Fredholm between these spaces.
We now explain how to generalize these two facts in the higher dimensional setting.

Consider a point $q$ which lies in $\calT^{(m)} \setminus \calT^{(m-1)}$. In other words, $q$ lies in the interior of one of the
lower dimensional convex spherical polytopes which arise as corners of the original polytopes $P_j$.  The interior to
any polytope $P_j$ which contains $q$ on its boundary is modelled as a sector inside the $(n-m)$-dimensional Euclidean
fiber of the normal bundle $N_q \calT^{(m)}$. This sector is described as an exact cone $C(Q)$, where $Q$ is
a convex spherical polytope in $S^{n-m-1}$.   The Euclidean Laplacian on this cone takes the form
\[
\del_\rho^2 + \frac{n-m-1}{\rho} \del_\rho + \frac{1}{\rho^2} \Delta_Q,
\]
where (abusing language slightly for the moment), the variable $\rho$ now denotes the radial variable in this cone. If the
eigenvalues of $\Delta_Q$ with Dirichlet boundary conditions at $\del Q$ are written as $\{\mu_j^{(n-m-1)}\}$, then
the indicial roots of this operator are equal to
\[
(\gamma_j^{(n-m-1)})^{\pm} = \frac{2 + m-n}{2} \pm \frac12 \sqrt{ (n-m-2)^2+ 4\mu_j^{(n-m-1)}}.
\]
These are the expected rates of growth or decay of solutions to $\Delta u = 0$ (or $L_\lambda u = 0$) at $\calT^{(m)}$. 
We have embellished the $\mu$'s and $\gamma$'s with superscripts to indicate the relevant dimensions.
Now observe that due to its convexity, $Q$ lies strictly inside a hemisphere of $S^{n-m-1}$, hence $\mu_0^{(n-m-1)} > n-m-1$. With
a bit of arithmetic, this leads to the estimate that
\[
(\gamma_0^{(n-m-1)})^- < - (n-m-1) < 1 < (\gamma_0^{(n-m-1)})^+.
\]

We can now assemble all this information.  The first point is that we can define iterated edge H\"older spaces with multi-weights
\[
\rho_{n-2}^{\nu^{(n-2)}} \rho_{n-3}^{\nu^{(n-3)}} \cdots \rho_0^{\nu^{(0)}}\calC^{k,\alpha}_{\ie}(S^n; \calI)
\]
to consist of all $\calC^{k,\alpha}_{\ie}$ sections of $\calI$ which decay like $\rho_m^{\nu^{(m)}}$ at $\calT^{(m)}$. Here the $\rho_m$
are smoothed versions of the distance functions to $\calT^{(m)}$, and the $\nu^{(m)}$ are the weight parameters at this stratum.
For simplicity, we write the weight prefactor as $\vec{\rho}^{\, \vec{\nu}}$.  The indicial root calculations above can be used (cf.\ \cite{AMU}) to prove the following two facts:
\begin{lemma} The map
\[
L_\lambda:  \vec{\rho}^{\, \vec{\nu}} \calC^{2,\alpha}_{\ie,\Gamma}(S^n) \longrightarrow \vec{\rho}^{\, \vec{\nu}-2}\calC^{0,\alpha}_{\ie,\Gamma}(S^n)
\]
is Fredholm whenever $(\gamma_m^{(n-m-1)})^- < \nu_m < (\gamma_m^{(n-m-1)})^+$ for all $m$ (this is the analogue of $\nu\in (-3/2,3/2)$).
Here $\vec{\nu}-2$ is the multi-weight obtained by subtracting $2$ from each entry of $\vec{\nu}$. 
Furthermore, if $L_\lambda u = 0$ for $u \in L^2_\Gamma$, then $u \in \vec{\rho}^{\, \vec{\nu}}\calC^{2,\alpha}_{\ie,\Gamma}$ for any $\vec{\nu}$
in this range. 
\end{lemma}
We then also have:
\begin{lemma}
If each $\nu_m \in (1,  (\gamma_0^{n-m-1})^+)$, then 
\[
\calH: \calU \longrightarrow \vec{\rho}^{\, \vec{\nu}-2}\calC^{0,\alpha}_{\ie,\Gamma}(S^n)
\]
is smooth, where $\calU$ is a neighborhood of $0$ in $\vec{\rho}^{\, \vec{\nu}}\calC^{2,\alpha}_{\ie,\Gamma}(S^n)$. 
\end{lemma}

Putting these two facts together, and using the bifurcation theorems in exactly the same way, we obtain the other main result:
\begin{theorem}\label{thm:bifurcation2}
Suppose there is an appropriate action by a discrete group $\Gamma$ on $S^n$. Then for each $j\geq 1$, there is an $\eps_j>0$ such that for 
$\lambda\in (\lambda_j,\lambda_j+\eps_j)$ there exist non-trivial $\Gamma$-equivariant solutions $u_{j,\lambda}$ to the minimal 
surface equation on $S^n$ in the warped product $(S^n \times (-1,1), g_\lambda)$;
these solutions bifurcate off the main branch $u\equiv 0$.
\end{theorem}

We conclude this section with an easy observation.
\begin{corollary}
Let $(X, g)$ denote an arbitrary compact $(n+1)$-manifold.  Then there is a (potentially large) deformation of the metric $g$ to a new metric $g'$ such that
$(X, g')$ admits compact minimal hypersurfaces which are multi-valued graphs over a sphere $S^n \subset X$.
\end{corollary}
For the proof, simply note that we can choose a new metric $g_1$ which has a neighborhood which is isometric to the product $S^n \times (-1,1)$.
We then let $g_\lambda'$ denote a family of metrics on $X$ which restrict to this neighborhood and are isometric to the warped product
metric $g_\lambda$ there.   Following the steps above, we can then construct branched minimal submanifolds in this neighborhood.

\bibliographystyle{alpha} 
\bibliography{bilbo}

\end{document}